\RequirePackage{fix-cm}
\documentclass[smallextended]{svjour3}       
\smartqed  
\usepackage{graphicx}
\usepackage{xcolor}
\usepackage[english]{babel}
\usepackage{amsmath}

\usepackage{bbm}
\usepackage{amssymb}
\newtheorem{thm}{Theorem}[section]
\numberwithin{thm}{section}
\newtheorem{cor}[thm]{Corollary}
\newtheorem{lem}[thm]{Lemma}
\newtheorem{prop}[thm]{Proposition}
\newtheorem{defn}[thm]{Definition}
\newtheorem{rem}[thm]{Remark}
\numberwithin{equation}{section}

 \newcommand\R{\mathbb{R}}

 \newcommand\E{\mathbb{E}}
 \newcommand\HH{\mathcal{H}}

 \DeclareMathOperator*{\var}{Var}

 \DeclareMathOperator*{\Dom}{Dom}

\begin{document}

\title{Local times for systems of non-linear stochastic heat equations
}


\author{Brahim Boufoussi         \and
        Yassine Nachit 
}


\institute{Brahim Boufoussi \at
              Department of Mathematics, Faculty of Sciences Semlalia, Cadi Ayyad University, 2390 Marrakesh, Morocco \\
              \email{boufoussi@uca.ac.ma}           
           \and
           Yassine Nachit \at
              Department of Mathematics, Faculty of Sciences Semlalia, Cadi Ayyad University, 2390 Marrakesh, Morocco \\
              \email{yassine.nachit.fssm@gmail.com}
}

\date{Received: date / Accepted: date}

\maketitle

\begin{abstract}
We consider $u(t,x)=(u_1(t,x),\cdots,u_d(t,x))$ the solution to a system of non-linear stochastic heat equations in spatial dimension one driven by a $d$-dimensional space-time white noise. We prove that, when $d\leq 3$, the local time $L(\xi,t)$ of $\{u(t,x)\,,\;t\in[0,T]\}$ exists and $L(\bullet,t) $ belongs a.s. to the Sobolev space $ H^{\alpha}(\R^d)$ for $\alpha<\frac{4-d}{2}$, and when $d\geq 4$, the local time does not exist. We also show joint continuity and establish H\"{o}lder conditions
for the local time of $\{u(t,x)\,,\;t\in[0,T]\}$. These results are then used to investigate the irregularity
of the coordinate functions of $\{u(t,x)\,,\;t\in[0,T]\}$.  Comparing to similar results obtained for the linear stochastic heat equation (i.e., the solution is Gaussian), we believe that our results are sharp. Finally, we get a sharp estimate for the partial derivatives of the joint density of $(u(t_1,x)-u(t_0,x),\cdots,u(t_n,x)-u(t_{n-1},x))$, which is a new result and of independent interest.
\end{abstract}

\section{Introduction and main results}
We consider the following system of non-linear stochastic heat equations
\begin{equation}\label{SHE}
  \frac{\partial u_k}{\partial t}(t,x) = \frac{\partial^2 u_k}{\partial x^2}(t,x)+b_k(u(t,x))+\sum_{l=1}^{d}\sigma_{k,l}(u(t,x))\dot{W}^l(t,x),
\end{equation}
with Neumann boundary conditions
$$u_k(0,x)=0,\qquad \frac{\partial u_k(t,0)}{\partial x}=\frac{\partial u_k(t,1)}{\partial x}=0,$$
for $1\leq k\leq d$, $t\in [0,T]$, $x\in [0,1]$, where $u:=(u_1,\cdots,u_d)$. Here, $\dot{W}=(\dot{W}^1,\cdots,\dot{W}^d)$ is a vector of
$d$-independent space-time white noises on $[0,T]\times [0,1]$.
 We put $b=(b_k)_{1\leq k\leq d} $ and $\sigma=(\sigma_{k,l})_{1\leq k, l\leq d}$. Following Walsh \cite{Walsh},  we will give, in Section \ref{Stochastic heat equation}, a rigorous formulation of the formal equation \eqref{SHE}. Let us state the following hypotheses on the coefficients $\sigma_{k,l}$ and $b_k$ of the system of non-linear stochastic heat equations $\eqref{SHE}$:\\
\textbf{A1}\; For all $1\leq k,l\leq d$,  the functions $\sigma_{k,l}$ and $b_k$ are bounded and  infinitely differentiable  such that the partial
derivatives of all orders are bounded.\\
\textbf{A2}\; The matrix $\sigma$ is uniformly elliptic i.e., there exists $\rho>0$ such that for all $x\in \R^d$ and $z\in \R^d$ with $\|z\|=1$, we have $\|\sigma(x)z\|^2\geq \rho^2$ (where $\|\cdot\|$ is the Euclidean norm on $\R^d$).

The objective of this paper is to investigate existence and regularity of the local time of $\{u(t,x)\,;\;t\in [0,T]\}$ the solution to Eq. \eqref{SHE}. The challenge to study local times of $\{u(t,x)\,;\;t\in [0,T]\}$ is twofold: on one hand, $\{u(t,x)\,;\;t\in [0,T]\}$ is neither a Gaussian process nor a stable process in general, on the other hand, the coordinate processes $u_1,\cdots,u_d$ are not independent. As far as we know, no one has studied the local time of an $\R^d$-valued process $X$ with coordinate processes $X_1,\cdots,X_d$ which are not independent, even in the Gaussian case. The only local times results that we are aware of for non-Gaussian processes are \cite{LouOuyang,Nourdin}.
By conditional Malliavin calculus approach, Lou and Ouyang have established in \cite{LouOuyang}
an upper bound of Gaussian type for the partial derivatives of the n-point joint density of the solution to a stochastic differential equation driven
by fractional Brownian motion. Their result is similar to Theorem \ref{estimation density}(b) below. Furthermore, they have used that result as an alternative to the classical local nondeterminism condition (LND for short), which is often used to investigate local times of Gaussian random fields -- for more details on the LND condition one can see \cite{Berman73}. Due to this, the authors in  \cite{LouOuyang} have proved the existence
and regularity of the local times of stochastic differential equations driven by fractional Brownian motions.
Moreover, in \cite{Nourdin}, Kerchev \textit{et al.} have investigated the existence and regularity of the local time of the Rosenblatt process via Berman's method.
Their proof is based
on a spectral analysis of arbitrary linear combinations of integral operators, which derive from the representation of the Rosenblatt process as an element in the second chaos.
The approach adopted here is quite different from the two previous works. Indeed, we mainly use two new tools: we introduce some techniques in the Malliavin calculus for adapted stochastic processes (see Section \ref{Malliavin calculus for adapted processes}) and also a condition involving a local estimation of the characteristic  function of the increments of a given process, which we call $\alpha$-local nondeterminism ($\alpha$-LND for short),
see the definition in section \ref{The local times}.  The role of the $\alpha$-LND condition in our investigation is the same as that of the well-known local nondeterminism condition used in the Gaussian framework. As part of our arguments, by using those techniques of the Malliavin calculus, we show that the systems of non-linear stochastic heat equations satisfy the $\frac{1}{4}$-LND condition. With this in mind and through Berman's method, we conclude the existence and joint continuity of the local time. Roughly speaking, we believe that the approach presented here can be used to investigate the local times of adapted stochastic processes that are smooth in the Malliavin sense. The main result of our studying is summarized as follows:
\begin{thm}\label{local time SHE}
  Let $u(t,x)$ be the solution to Eq. \eqref{SHE}, and $x\in (0,1)$ be fixed.
  \begin{description}
    \item[(i)] Almost surely, when $d\leq 3$, the local time $L(\xi,t)$ of the process $\{u(t,x)\,,\;t\in[0,T]\}$ exists for any fixed $t$, moreover, $L(\bullet,t) $ belongs to the Sobolev space $ H^{\alpha}(\R^d)$ of index $\alpha<\frac{4-d}{2}$; and when $d\geq 4$, the local time does not exist in $L^2(\mathbb{P}\otimes \lambda_d)$ for any $t$, where $\lambda_d$ is the Lebesgue measure on $\R^d$.
    \item[(ii)] Assume $d\leq 3$, the local time of the process $\{u(t,x)\,,\;t\in[0,T]\}$ has a version, denoted by $L(\xi,t)$, which is a.s. jointly continuous in $(\xi,t)$, and which is $\gamma$-H\"{o}lder continuous in $t$, uniformly in $\xi$, for all $\gamma<1-\frac{d}{4}$: there exist two random variables $\eta$ and $\delta$ which are almost surely finite and positive such that
        $$\sup_{\xi\in \R^d}|L(\xi,t)-L(\xi,s)|\leq \eta\, |t-s|^{\gamma},$$
        for all $s,t\in [0,T]$ such that $|t-s|<\delta$.
  \end{description}
\end{thm}
As a consequence, one can get a result on the behavior of the coordinate functions of the solution to Eq. \eqref{SHE}.
\begin{cor}\label{cor nonHolder}
  Let $u(t,x)$ be the solution to Eq. \eqref{SHE}. Assume $d\leq 3$. Then for each $x\in (0,1)$, almost surely, all coordinate functions of $\{u(t,x)\,,\;t\in[0,T]\}$
   are nowhere H\"{o}lder continuous of order greater than $\frac{1}{4}$.
\end{cor}

When we were investigating local times, we got the following theorem which is interesting in its own right.

Let $x\in (0,1)$ be fixed, and let $\pi_n=(t_1,\cdots,t_n)$ with $0=t_0<t_1<\cdots<t_n\leq T$. We denote by $p_{\pi_n,x}(\xi)$, where $\xi=(\xi_{j,l}\,,\; 1\leq j\leq n\,,\; 1\leq l\leq d)\in\R^{n\times d}$, the density of the $\R^{n\times d}$-valued random vector $(u(t_1,x)-u(t_0,x),\cdots,u(t_n,x)-u(t_{n-1},x))$, where $u(t_i,x)-u(t_{i-1},x)=(u_1(t_i,x)-u_1(t_{i-1},x),\cdots,u_d(t_i,x)-u_d(t_{i-1},x))$, for $i=1,\cdots,n$. Put also $p_{s,t,x}(y)$ the density of the $\R^d$-valued random vector $(u_1(t,x)-u_1(s,x),\cdots,u_d(t,x)-u_d(s,x))$. Set $\|\cdot\|$ for the Euclidean norm on $\R^d$. For all $\xi=(\xi_{j,l}\,,\; 1\leq j\leq n\,,\; 1\leq l\leq d)\in\R^{n\times d}$ and $m=(m_{j,l}\,,\; 1\leq j\leq n\,,\; 1\leq l\leq d)$ where $m_{j,l}$, for $j=1,\cdots, n$ and $l=1,\cdots, d$, are nonnegative integers, we introduce
$$\partial^{m}_{\xi}=\prod_{j=1}^{n}\prod_{l=1}^{d}\left(\frac{\partial}{\partial\xi_{j,l}}\right)^{m_{j,l}}.$$
\begin{thm}\label{estimation density}
 Assume \normalfont{\textbf{A1}} and \normalfont{\textbf{A2}}. Then we get the following:
 \begin{description}
   \item[(a)]\label{Gaussian lower bound} There exists a constant $c>0$ such that for any $x\in (0,1)$, $0\leq s<t\leq T$, and $y\in \R^{d}$,
   \begin{equation}\label{lower bound for the density}
     p_{s,t,x}(y)\geq \frac{c}{(t-s)^{d/4}}\exp\left(-\frac{\|y\|^2}{c(t-s)^{1/2}}\right).
   \end{equation}
   \item[(b)]\label{Gaussian upper bound} Let $n$ be a positive integer and $m_{i,k}$, for $i=1,\cdots, n$ and $k=1,\cdots, d$, be nonnegative integers. Then, there exists a positive constant $c$ (may depend on $n$ and $m_{i,k}$) such that for all $x\in (0,1)$,  $\pi_n=(t_1,\cdots,t_n)$ with $0=t_0<t_1<\cdots<t_n\leq T$, and  $\xi=(\xi_{j,l}\,,\;1\leq j\leq n\,,\; 1\leq l\leq d)\in \R^{n\times d}$,
       \begin{equation}\label{upper bound derivative dennsity}
            \left|\partial^{m}_{\xi} \,p_{\pi_n,x}(\xi)\right|
               \leq c\,\prod_{i=1}^{n} \frac{1}{(t_i-t_{i-1})^{(d+\sum_{k=1}^{d}m_{i,k})/4}}\exp\left(-\frac{\|\xi_i\|^2}{c\,(t_i-t_{i-1})^{1/2}}\right),
       \end{equation}
       where $\xi_i=(\xi_{i,1},\cdots,\xi_{i,d})$  and $m=(m_{i,k}\,,\; 1\leq i\leq n\,,\; 1\leq k\leq d)$.
 \end{description}
\end{thm}

In \cite{DalangKhoshnevisanNualartmultiplicative}, using Malliavin calculus techniques, Dalang \textit{et al}. have established a Gaussian-type
lower bound for the one-point density of the solution $u(t,x)$ and a Gaussian-type upper
bound for the two-point density of $(u(s,y),u(t,x))$, in order to get upper and lower bounds for the hitting probabilities of
the process $\{u(t,x)\,;\;t\in \R^+,\;x\in [0,1]\}$. Therefore, we see that Theorem \ref{estimation density} generalizes in some sense   \cite[Theorem 1.1]{DalangKhoshnevisanNualartmultiplicative}.

Now, recall that a subset $A$ of $\R^d$ is polar for an $\R^d$-valued stochastic process $(X_t)_{t\in [0,T]}$ if
$\mathbb{P}\left[X_t\in A \text{ for some } t\right]=0.$
From \cite{DalangKhoshnevisanNualartmultiplicative} we know that singletons are not polar for $\{u(t,x)\,,\;t\in[0,T]\}$ when $d\leq 3$, and are
polar for $d\geq 5$. The critical case of dimension $d=4$ is open. We think that there is a deep connection between the polarity and the existence of the local time. Indeed, it is easy to prove that if almost all  singletons are polar, then the local time does not exist. As a consequence the local time does not exist for $\{u(t,x)\,,\;t\in[0,T]\}$ when $d\geq 5$, but in section \ref{42}, we will give another proof of that, and we will even show that the local time does not exist in the critical dimension $d=4$.

Finally, let us briefly explain that the $\alpha$-LND property (see Definition \ref{alphha LND}) can be a consequence of the integration by parts formula (see Proposition \ref{prop integration by parts}) as follows:
$$\E\left[\partial^{k}_{y}\, e^{i\sum_{j=1}^{n}\left<\xi_j,u(t_j,x)-u(t_{j-1},x)\right>}\right]=\E\left[e^{i\sum_{j=1}^{n}\left<\xi_j , u(t_j,x)-u(t_{j-1},x)\right>}H^{\beta}_{\pi_n}(Z\,,\,1)\right],$$
where $\left<\cdot,\cdot\right>$ is the Euclidean inner product on $\R^d$, $k=(k_{j,l},\;1\leq j\leq n,\,1\leq l\leq d)$, $\xi=(\xi_{j,l},\;1\leq j\leq n,\,1\leq l\leq d)$, and $Z=(u(t_1,x)-u(t_0,x),\cdots,u(t_n,x)-u(t_{n-1},x)).$ Our observation is that by the above equality we have,
$$\prod_{h=1}^{n}\prod_{l=1}^{d}|\xi_{h,l}|^{k_{h,l}}\left|\E\left[e^{i\sum_{j=1}^{n}\left<\xi_j,u(t_j,x)-u(t_{j-1},x)\right>}\right]\right|\leq \E\left[\left|H^{\beta}_{\pi_n}(Z\,,\,1)\right|\right].$$
One of the main technical efforts in this paper is to estimate $\E[|H^{\beta}_{\pi_n}(Z\,,\,1)|]$.

The rest of the paper is arranged as follows. In the second section, we give some preliminary results on local times, classical Malliavin calculus, and the stochastic heat equation. The third section is devoted to introducing some new tools in the Malliavin calculus that we will use in the fourth and fifth section in order to prove respectively  Theorem \ref{estimation density} and Theorem \ref{local time SHE}.

Finally,  we mention that constants in our proofs may change from line to line.

\section{Preliminaries}
\subsection{The local times}\label{The local times}
This section is devoted to briefly give some aspects of the theory of local times. For more information on the subject, we refer to the classical paper of  Geman and Horowitz  \cite{GemanHorowitz}.

Let $(\theta_t)_{t\in [0,T]}$ be a Borel function with values in $\R^d$. For any Borel set $B\subseteq [0,T]$, the occupation measure of $\theta$ on $B$ is given by the following measure on $\R^d$:
$$\nu_B(\bullet)=\lambda\{t\in B\,;\;\theta_t\in \bullet\},$$
where $\lambda$ is the Lebesgue measure. When $\nu_B$ is absolutely continuous with respect to $\lambda_d$ (the Lebesgue measure on $\R^d$), we say that the local time of $\theta$ on $B$ exists and it is defined, $L(\bullet,B)$, as the Radon-Nikodym derivative of $\nu_B$ with respect to $\lambda_d$, i.e., for almost every $x$,
$$L(x,B)=\frac{d\nu_B}{d\lambda_d}(x).$$
In the above, we call $x$ the space variable and $B$ the time variable. We write $L(x,t)$ (resp. $L(x)$) instead of $L(x,[0,t])$ (resp. $L(x,[0,T])$).

The local time satisfies the following occupation formula: for any Borel set $B\subseteq [0,T]$, and for every measurable bounded function $f:\R^d\to \R$,
$$\int_{B}f(\theta_s)ds=\int_{\R^d}f(x)L(x,B)dx.$$

The deterministic function $\theta$ can be chosen to be the sample path of  an $\R^d$-valued separable stochastic process $X=(X_t)_{t\in [0,T]}$ with $X_0=0$ a.s. In this regard, we say that the process $X$ has a local time (resp. square integrable local time) if for almost all $\omega$, the trajectory $t\mapsto X_t(\omega)$ has a local time (resp. square integrable local time).

We investigate the local time via Berman's approach. The idea is to derive properties of $L(\bullet,B)$ from the integrability properties of the Fourier transform of the sample paths of $X$.

Let us state the following hypotheses on the integrability of the characteristic function of $X$:\\
\textbf{B1}\;$$\int_{\R^d}\int_{0}^{T}\int_{0}^{T}\E\left[e^{i\left<u,X_t-X_s\right>}\right]dt\,ds\,du<\infty.$$
\textbf{B2}\; For every even integer $m\geq 2$,
$$\int_{(\R^d)^m}\int_{[0,T]^m}\left|\E\left[\exp\left(i\sum_{j=1}^{m}\left<u_j,X_{t_j}\right>\right)\right]\right|\prod_{j=1}^{m}dt_j\prod_{j=1}^{m}du_j<\infty.$$

Recall the following crucial result in \cite{Berman69a} (see also \cite[Lemma 3.1]{Berman69b}):
\begin{thm}
  Assume \normalfont{\textbf{B1}}. Then the process $X$ has a square integrable local time. Moreover, we have almost surely, for all Borel set $B\subseteq [0,T]$, and for almost every $x$,
  \begin{equation}\label{local time L2}
    L(x,B)=\frac{1}{(2\pi)^d}\int_{\R^d}e^{-i\left<u,x\right>}\int_B e^{i\left<u,X_t\right>}dt\,du.
  \end{equation}
\end{thm}

In \eqref{local time L2}, $L(x,B)$ is not a stochastic process. Following Berman \cite{Berman69b} we construct a version of the local time, which is a stochastic process.\\
The below theorem is given in Berman \cite[Theorem 4.1]{Berman69b} for $d=1$ and $m=2$, so we will omit the proof.
\begin{thm}
  Assume \normalfont{\textbf{B1}} and \normalfont{\textbf{B2}}. Put for all integer $N\geq 1$,
  $$L_N(x,t)=\frac{1}{(2\pi)^d}\int_{[-N,N]^d}e^{-i\left<u,x\right>}\int_{0}^{t} e^{i\left<u,X_s\right>}ds\,du.$$
  Then there exists a stochastic process $\tilde{L}(x,t)$ separable in the $x$-variable, such that for each even integer $m\geq 2$,
  \begin{equation}\label{L tilde}
    \lim_{N\to \infty}\sup_{(x,t)\in \R^d\times [0,T]}\E\left[|L_N(x,t)-\tilde{L}(x,t)|^m\right]=0.
  \end{equation}
\end{thm}

\begin{thm}[Theorem 4.3 in \cite{Berman69b}]
  Let $\tilde{L}(x,t)$ be given by $\eqref{L tilde}$.
  If the stochastic process $\{\tilde{L}(x,t),\, x\in\R^d\}$ is almost surely continuous, then it is a continuous (in the $x$-variable) version of the local time on $[0,t]$.
\end{thm}

In order to overcome the problem caused by the dependence of the negligible sets on $x$, $t$, and $\omega$, we will look for a jointly continuous version of the local time. For this end, we have by \eqref{L tilde}, for all $x,y\in \R^d$, $t, h\in [0,T]$ such that $t+h\in [0,T]$, and even integer $m\geq 2$,
\begin{equation}\label{E Lxkth Lxth Lxkt Lxt}
    \begin{split}
        &\E[\tilde{L}(x+y,t+h)-\tilde{L}(x,t+h)-\tilde{L}(x+y,t)+\tilde{L}(x,t)]^m=\frac{1}{(2\pi)^{md}} \\
         & \times\int_{(\R^d)^m}\int_{[t,t+h]^m}\prod_{j=1}^{m}\left(e^{-i\left<u_j,x+y\right>}-e^{-i\left<u_j,x\right>}\right)\E\left[e^{i\sum_{j=1}^{m}\left<u_j,X_{t_j}\right>}\right]\prod_{j=1}^{m}dt_j\prod_{j=1}^{m}du_j\\
         &=\frac{1}{(2\pi)^{md}}\int_{(\R^d)^m}\int_{[t,t+h]^m}\prod_{j=1}^{m}\left(e^{-i\left<v_j-v_{j+1},x+y\right>}-e^{-i\left<v_j-v_{j+1},x\right>}\right)\\
         &\qquad\qquad\qquad\qquad\qquad\qquad\times\E\left[e^{i\sum_{j=1}^{m}\left<v_j,X_{t_j}-X_{t_{j-1}}\right>}\right]\prod_{j=1}^{m}dt_j\prod_{j=1}^{m}dv_j,
    \end{split}
\end{equation}
and
\begin{equation}\label{E Lxth Lxt}
    \begin{split}
        &\E[\tilde{L}(x,t+h)-\tilde{L}(x,t)]^m\\
        &=\frac{1}{(2\pi)^{md}} \int_{(\R^d)^m}\int_{[t,t+h]^m}e^{-i\sum_{j=1}^{m}\left<u_j,x\right>}\,\E\left[e^{i\sum_{j=1}^{m}\left<u_j,X_{t_j}\right>}\right]\prod_{j=1}^{m}dt_j\prod_{j=1}^{m}du_j\\
        &=\frac{1}{(2\pi)^{md}} \int_{(\R^d)^m}\int_{[t,t+h]^m}e^{-i\left<v_1,x\right>}\,\E\left[e^{i\sum_{j=1}^{m}\left<v_j,X_{t_j}-X_{t_{j-1}}\right>}\right]\prod_{j=1}^{m}dt_j\prod_{j=1}^{m}dv_j,
    \end{split}
\end{equation}
where $t_0=0$ and the last equality in \eqref{E Lxkth Lxth Lxkt Lxt} (resp. \eqref{E Lxth Lxt}) holds by the following changes of variables:
$$u_j=v_j-v_{j+1},\qquad j=1,\cdots,m,\qquad\text{with}\qquad v_{m+1}=0.$$

For the purpose to estimate \eqref{E Lxkth Lxth Lxkt Lxt} and \eqref{E Lxth Lxt}, we need first to estimate the characteristic function $\E\left[e^{i\sum_{j=1}^{m}\left<v_j,X_{t_j}-X_{t_{j-1}}\right>}\right]$. Therefore, we introduce the
following new condition, called $\alpha$-local nondeterminism ($\alpha$-LND for short), which involves local estimation of the characteristic function of the
increments:
\begin{defn}\label{alphha LND}
  Let $X=(X_t)_{t\in [0,T]}$ be a stochastic process with values in $\R^d$ and $J$ a subinterval of $[0,T]$. $X$ is said to be $\alpha$-LND on $J$, if for every non-negative integers $m\geq 2$, and $k_{j,l}$, for $j=1,\cdots,m$, $l=1,\cdots,d$, there exist positive constants $c$ and $\varepsilon$, both may depend on $m$ and $k_{j,l}$, such that
  \begin{equation}\label{Ineq alphha LND}
        \left|\E\left[e^{i\sum_{j=1}^{m}\left<v_j,X_{t_j}-X_{t_{j-1}}\right>}\right]\right|
         \leq\frac{c}{\prod_{j=1}^{m}\prod_{l=1}^{d}|v_{j,l}|^{k_{j,l}}(t_j-t_{j-1})^{\alpha k_{j,l}}},
  \end{equation}
  for all $v_j=(v_{j, l}\,;\;1\leq l\leq d)\in (\R\setminus\{0\})^{d}$, for $j=1,\cdots,m$, and for every ordered points $t_1<\cdots<t_m$ in $J$ with $t_m-t_1<\varepsilon$ and $t_0=0$.
\end{defn}
\begin{rem}

\begin{enumerate}

\item  It is well-known that the local nondeterminism concept in the Gaussian context means that ``the value of the process at a given time
point is relatively unpredictable on the basis of a finite set of observations from the immediate past". In the Gaussian framework, Berman uses the conditional variance to express this. But unfortunately, he can't use the conditional variance beyond the Gaussian case because in a general framework  the conditional variance is not deterministic. So, Berman has introduced the concept of local $g$-nondeterminism for general processes by replacing the incremental variance, which is a measure of local unpredictability, by a measure of local
predictability, namely, the value of the incremental density function at the origin, see \cite[Definition 5.1]{Bermangeneral1983}.  Hence, the local $g$-nondeterminism concept reflects well his name. By the Fourier inversion theorem, it is easy to see that the condition in Definition \ref{alphha LND} implies  the local $g$-nondeterminism condition. On the other hand, Nolan has introduced the notion of \textit{characteristic function locally approximately independent increments} (see \cite[Definition 3.1]{nolan1989local}), which is equivalent in the Gaussian and stable framework to the classical LND condition. The condition in Definition \ref{alphha LND} ($d=1$) is an extension of Nolan's notion by replacing the characteristic functions $\left|\E\left[e^{ic_m u_{j}(X(t_j)-X(t_{j-1}))}\right]\right|$ in the right-hand side of \cite[Ineq. (3.3)]{nolan1989local} by $c\,|u_{j}|^{-k_{j}}(t_j-t_{j-1})^{-\alpha k_{j}}$. For all these reasons, we choose to call the condition in Definition \ref{alphha LND} by $\alpha$-local nondeterminism ($\alpha$-LND).

  \item  Let $d=1$ and $Y=(Y_t)_{t\in [0,T]}$ be a centred Gaussian process that satisfies the classical local nondeterminism (LND) property on $J$. By \cite[Lemma 2.3]{Berman73} we have for any $m\geq2$, there exist two positive constants $c_m$ and $\varepsilon$ such that for every ordered points $t_1<\cdots<t_m$ in $J$ with $t_m-t_1<\varepsilon$, and $(v_1,\cdots,v_m)\in \R^m\setminus\{0\}$,
  \begin{equation}\label{LND}
    \var\left(\sum_{j=1}^{m}v_j(Y_{t_j}-Y_{t_{j-1}})\right)\geq c_m \sum_{j=1}^{m}v^2_j\var\left(Y_{t_j}-Y_{t_{j-1}}\right).
  \end{equation}
  Assume also that there exists a positive constant $K$, such that for every $s,t\in J$ with $s<t$,

  \begin{equation}\label{minoration var}
    K(t-s)^{2\alpha}\leq\var\left(Y_{t}-Y_{s}\right).
  \end{equation}
  Hence $Y$ is $\alpha$-LND on $J$.
  \item Let $d>1$ and $Y^0=(Y^0_t)_{t\in [0,T]}$ be a real-valued centred Gaussian process that verifies the classical local nondeterminism (LND) property on $J$ (i.e. \eqref{LND}) and \eqref{minoration var}. Define $Y_t=(Y^1_t,\cdots,Y^d_t),$
  where $Y^1,\cdots,Y^d$ are independent copies of $Y^0$. Then $Y$ is $\alpha$-LND on $J$.
\end{enumerate}
\end{rem}

To end this section, we give the following theorem that explain clearly the fact that if a function's local time, $L(x,t)$, is  H\"{o}lder continuous in $t$ uniformly in $x$, then this has a major effect on the H\"{o}lder continuity of the function itself.
\begin{thm}[Theorem 8.7.1 in \cite{Adler}]\label{non Holder}
  Let $(\theta_s)_{s\in [0,T]}$ be a continuous function with values in $\R^d$, possessing a local time, $L(x,t)$, satisfying: there exist positive constants $M$ and $\delta$, such that
  $$\sup_{x\in \R^d}|L(x,t)-L(x,s)|\leq M |t-s|^{\beta},$$
  for every $s,t\in [0,T]$ with $|t-s|<\delta$. Then all coordinate functions of $\theta$ are nowhere H\"{o}lder continuous of order greater than $(1-\beta)/d$.
\end{thm}
\subsection{Classical Malliavin calculus}\label{Malliavin calculus}
In this section, we introduce some elements of Malliavin calculus, for more details we can refer to Nualart \cite{Nualart} (see also Sanz-Sol\'e \cite{Sanz-Sole}).
Let $(W^i(t,x),\;t\in [0,T],\;x\in [0,1])$, $i=1,\cdots,d$, be $d$-independent space-time white noises defined on a probability space $(\Omega,\mathcal{F},\mathbb{P})$, and set $W=(W^1,\cdots,W^d)$. For any $h=(h^1,\cdots,h^d)\in \mathcal{H}:=L^2([0,T]\times[0,1],\R^d)$, we put the Wiener integral $W(h)=\sum_{i=1}^{d}\int_{0}^{T}\int_{0}^{1}h^i(t,x)W^i(dx,dt).$
Denote by $\mathcal{S}$ the class of cylindrical random variables of the form $F=\varphi(W(h_1),\cdots,W(h_n))$, with arbitrary $n\geq 1$, $h_1,\cdots,h_n$ in $\mathcal{H}$, and $\varphi\in C^{\infty}_P(\R^n)$, where $C^{\infty}_P(\R^n)$ is the set of real-valued functions $\varphi$ such that $\varphi$ and all its partial derivatives have at most polynomial growth. Let $F\in \mathcal{S}$, we define the derivative of $F$ as the $d$-dimensional stochastic process $DF=\{D_{t,x}F=(D^{(1)}_{t,x}F,\cdots,D^{(d)}_{t,x}F),\;(t,x)\in [0,T]\times[0,1]\},$
where, for $l=1,\cdots,d$,
$$D^{(l)}_{t,x}F=\sum_{i=1}^{n}\frac{\partial \varphi}{\partial x_i}(W(h_1),\cdots,W(h_n))h_i^l(t,x).$$
More generally set $D_{\alpha}^kF=D_{\alpha_1}\cdots D_{\alpha_k}F$ for the derivative of order $k$ of $F$, where $\alpha=(\alpha_1,\cdots,\alpha_k),$ $\alpha_i=(t_i,x_i)\in [0,T]\times[0,1]$, $k$ an integer. For any $p,k\geq 1$, we denote by $\mathbb{D}^{k,p}$ the closure of the class $\mathcal{S}$ with respect to the norm $\|\cdot\|_{k,p}$ defined by
$$\|F\|_{k,p}=\left\{\E[|F|^p]+\sum_{j=1}^{k}\E[\|D^jF\|_{\HH^{\otimes j}}^p]\right\}^{\frac{1}{p}},$$
where
$$\|D^jF\|_{\HH^{\otimes j}}=\left\{\sum_{i_1,\cdots,i_j=1}^{d}\int_{0}^{T}dt_1\int_{0}^{1}dx_1\cdot\cdot\cdot\int_{0}^{T}dt_j\int_{0}^{1}dx_j\left(D_{t_1,x_1}^{(i_1)}\cdot\cdot\cdot D_{t_j,x_j}^{(i_j)}F\right)^2\right\}^{\frac{1}{2}}.$$
We put $\mathbb{D}^{\infty}=\cap_{p\geq 1}\cap_{k\geq 1}\mathbb{D}^{k,p}.$ Let $0\leq s<t\leq T$, we set $\HH_{s,t}:=L^2([s,t]\times[0,1],\R^d).$
For any $F,G\in \mathbb{D}^{1,p}$ we point out that
$$\left<DF\,,\,DG\right>_{\HH}=\sum_{l=1}^{d}\int_{0}^{T}\int_{0}^{1}D^{(l)}_{r,x}F\, D^{(l)}_{r,x}G\,drdx,$$
and
$$ \left<DF\,,\,DG\right>_{\HH_{s,t}}=\sum_{l=1}^{d}\int_{s}^{t}\int_{0}^{1}D^{(l)}_{r,x}F\, D^{(l)}_{r,x}G\,drdx.$$
Let $\mathcal{V}$ be a separable Hilbert space. We define $\mathcal{S}_{\mathcal{V}}$ as the class of $\mathcal{V}$-valued smooth random variables of the form $u=\sum_{j=1}^{n}F_j\,v_j$, where $F_j\in \mathcal{S}$ and $v_j\in \mathcal{V}$. Similarly, we can introduce the analogous spaces $\mathbb{D}^{k,p}(\mathcal{V})$ and $\mathbb{D}^{\infty}(\mathcal{V})$, and the related norm $\|\cdot\|_{k,p,\mathcal{V}}$ defined by
$$\|u\|_{k,p,\mathcal{V}}=\left\{\E[\|u\|_{\mathcal{V}}^p]+\sum_{j=1}^{k}\E[\|D^ju\|_{\HH^{\otimes j}\otimes\mathcal{V}}^p]\right\}^{\frac{1}{p}}.$$

We denote by $\delta$ the Skorohod integral, which is defined as the adjoint of the operator $D$. $\delta$ is an unbounded operator on $L^2(\Omega,\mathcal{H})$ taking values in $L^2(\Omega)$. The domain of $\delta$, denoted by $\Dom(\delta)$, is the set of $u\in L^2(\Omega,\mathcal{H})$ such that there exists a constant $c>0$ satisfying $|\E[\left<DF,u\right>_{\HH}]|\leq c\|F\|_{0,2}$, for every $F\in \mathbb{D}^{1,2}$. Let $u\in \Dom(\delta)$, then $\delta(u)$ is the unique element of $L^2(\Omega)$ characterized by the duality formula
\begin{equation}\label{duality}
  \E[F\delta(u)]=\E\left[\sum_{l=1}^{d}\int_{0}^{T}\int_{0}^{1}D^{(l)}_{t,x}F\,u_l(t,x)\,dtdx\right],\quad \text{for all} \;F\in\mathbb{D}^{1,2}.
\end{equation}

We will use the following estimate for the norm $\|\cdot\|_{k,p}$ of the Skorohod integral.
\begin{prop}[\cite{Nualart}, Proposition 1.5.7]
  The divergence operator $\delta$ is continuous from $\mathbb{D}^{k+1,p}(\HH)$ to $\mathbb{D}^{k,p}$ for every $p>1,$ $k\geq 0$. Therefore, there exists a constant $c_{k,p}>0$ such that for any $u\in \mathbb{D}^{k+1,p}(\HH)$,
  \begin{equation}\label{continuite of the divergence operator}
    \|\delta(u)\|_{k,p}\leq c_{k,p} \|u\|_{k+1,p,\HH}.
  \end{equation}
\end{prop}

\subsection{Stochastic heat equation}\label{Stochastic heat equation}
First note that Eq. \eqref{SHE} is formal, it can be formulated rigorously as follows (Walsh \cite{Walsh}): let $\mathcal{B}([0,1])$ be the Borel $\sigma$-algebra on $[0,1]$ and $W^l=(W^l(t,A),\;t\in [0,T],\; A\in \mathcal{B}([0,1]))$, where $l=1,\cdots,d$, be independent space-time white noises, defined on a complete probability space $(\Omega,\mathcal{F},\mathbb{P})$, i.e, $W^1,\cdots,W^d$ are independent and $W^l$ is a centred Gaussian process with covariance function given by
$$\E[W^l(t,A)W^l(s,B)]=(t\wedge s)\lambda(A\cap B),$$
for $1\leq l\leq d$, $t,s\in [0,T]$, and $A,B\in \mathcal{B}([0,1])$, where $\lambda$ is the Lebesgue measure. Set $W=(W^1,\cdots,W^d)$ and $W(t,x)=W(t,[0,x])$. For $t\in[0,T]$, let $\mathcal{F}_t=\sigma\{W(s,A),\;s\in [0,t],\;A\in \mathcal{B}([0,1])\}\vee \mathcal{N}$, where $\mathcal{N}$ is the collection of $\mathbb{P}$-null sets. We say that a process $u=\{u(t,x),\; t\in [0,T],\; x\in [0,1]\}$ is adapted to the filtration $(\mathcal{F}_t)_{0\leq t\leq T}$ if $u(t,x)$ is $\mathcal{F}_t$-measurable for each $(t,x)\in [0,T]\times [0,1]$. A mild solution of \eqref{SHE} is a jointly measurable $\R^d$-valued process $u=(u_1,\cdots,u_d)$ such that $u$ is adapted to $(\mathcal{F}_t)_{0\leq t\leq T}$ and for any $k\in \{1,\cdots,d\}$, $t\in [0,T]$, and $x\in [0,1]$,
\begin{equation}\label{solution SHE}
  \begin{split}
     u_k(t,x) &= \int_{0}^{t}\int_{0}^{1}G_{t-r}(x,v)\sum_{l=1}^{d}\sigma_{k,l}(u(r,v))W^l(dr,dv) \\
       &\qquad + \int_{0}^{t}\int_{0}^{1}G_{t-r}(x,v)b_k(u(r,v)) dv\, dr.
  \end{split}
\end{equation}
Here, The stochastic integral shall be read as in \cite{Walsh} and $G_t(x,y)$ denotes the Green kernel of the heat equation with Neumann boundary
conditions (see \cite[Chap. 3]{Walsh}). In this paper, We are not interested in the explicit form of $G_t(x,y)$, we will need just the following three properties:
\begin{itemize}
  \item The symmetry \cite[(3.6)]{Walsh}: $G_t(x,y)=G_t(y,x)$;
  \item The semi-group property \cite[(3.6)]{Walsh}: $\int_{0}^{1}G_t(x,y)G_s(y,z)dy=G_{t+s}(x,z)$;
  \item The Gaussian-type bound \cite[(A.1)]{BallyPardoux}:
  \begin{equation}\label{bound Green kernel}
    c_1\,\phi_{t-s}(x-y)\leq G_{t-s}(x,y)\leq c_2\,\phi_{t-s}(x-y),
  \end{equation}
  where $\phi_{t-s}(x-y)=\frac{1}{\sqrt{2\pi (t-s)}}\exp\left(-\frac{|x-y|^2}{2(t-s)}\right)$.
\end{itemize}

Modifying  the results from \cite{Walsh} to the case $d\geq 1$, one can show that (when $\sigma_{k,l}$ and $b_k$ are Lipschitz),  there exists a unique continuous $\R^d$-valued process $u=\{u(t,x),\; t\in [0,T],\; x\in [0,1]\}$ adapted to $(\mathcal{F}_t)_{0\leq t\leq T}$ which is the mild solution of \eqref{SHE}. Furthermore, it is shown in \cite{BallyMilletSanz} that for any  $0\leq s\leq t\leq T$, $x,y\in [0,1]$, and $p>1$,
\begin{equation}\label{kolmoggorov SHE}
  \E[|u(t,x)-u(s,y)|^p]\leq C\left[|t-s|^{1/2}+|x-y|\right]^{p/2}.
\end{equation}
Therefore, $u$ is $(\frac{1}{4}-\varepsilon)$-H\"{o}lder continuous in $t$ and $(\frac{1}{2}-\varepsilon)$-H\"{o}lder continuous in $x$.

Now we are concerned with the study of the Malliavin differentiability of $u$ and the equations fulfilled by its Malliavin derivatives (Proposition \ref{prop derivatives SHE}). We refer to Bally and Pardoux \cite[Proposition 4.3, (4.16), (4.17)]{BallyPardoux} for a complete proof in dimension one. In this paper, we work coordinate by coordinate, therefore the below proposition follows in the same way, and its proof is then omitted.
\begin{prop}[Proposition 4.1 \cite{DalangKhoshnevisanNualartmultiplicative}]\label{prop derivatives SHE}
  Assume \normalfont{\textbf{A1}}. Then for any $t\in [0,T]$ and $x\in [0,1]$ we have $u(t,x)\in (\mathbb{D}^{\infty})^d$. Furthermore, its iterated derivative satisfies for $n\geq1$ and all $r_1,\cdots,r_n\in[0,T]$ such that $r_1\vee \cdots\vee r_n<t$,
  \begin{equation*}
    \begin{split}
      & D^{(i_1)}_{r_1,v_1} \cdots D^{(i_n)}_{r_n,v_n}(u_k(t,x)) \\
       &= \sum_{p=1}^{d}G_{t-r_p}(x,v_p)\left(D^{(i_1)}_{r_1,v_1} \cdots D^{(i_{p-1})}_{r_{p-1},v_{p-1}}D^{(i_{p+1})}_{r_{p+1},v_{p+1}}\cdots D^{(i_n)}_{r_n,v_n}(\sigma_{k,l_p}(u(r_p,v_p)))\right)   \\
         &\qquad+\sum_{l=1}^{d}\int_{r_1\vee \cdots\vee r_n}^{t}\int_{0}^{1}G_{t-\tau}(x,z)\prod_{q=1}^{n}D^{(i_q)}_{r_q,v_q}(\sigma_{k,l}(u(\tau,z)))W^l(d\tau,dz)  \\
         & \qquad\qquad+\int_{r_1\vee \cdots\vee r_n}^{t}\int_{0}^{1}G_{t-\tau}(x,z)\prod_{q=1}^{n}D^{(i_q)}_{r_q,v_q}(b_k(u(\tau,z)))dz\,d\tau,
    \end{split}
  \end{equation*}
  and when $t \leq r_1\vee \cdots\vee r_n$ we have $D^{(i_1)}_{r_1,v_1} \cdots D^{(i_n)}_{r_n,v_n}(u_k(t,x))=0$. Finally, for any $p>1$,
  \begin{equation}\label{sup derivative SHE}
    \sup_{(t,x)\in [0,T]\times[0,1]}\E\left[\|D^n u_k(t,x)\|^p_{\mathcal{H}^{\otimes n}}\right]<\infty.
  \end{equation}
\end{prop}
\begin{rem}
Point out that, in particular, the first-order Malliavin derivative fulfils, for $r<t$,
\begin{equation}\label{first derivative SHE}
  D^{(i)}_{r,v}(u_k(t,x))= G_{t-r}(x,v)\sigma_{k,i}(u(r,v))+a_k(i,r,v,t,x),
\end{equation}
where
\begin{equation}\label{akirvtx}
    \begin{split}
       a_k(i,r,v,t,x) &= \sum_{l=1}^{d}\int_{r}^{t}\int_{0}^{1}G_{t-\tau}(x,z)D_{r,v}^{(i)}(\sigma_{k,l}(u(\tau,z)))W^l(d\tau,dz) \\
         &\qquad+ \int_{r}^{t}\int_{0}^{1}G_{t-\tau}(x,z)D_{r,v}^{(i)}(b_k(u(\tau,z)))dz\,d\tau,
    \end{split}
\end{equation}
and
\begin{equation}\label{derivative order 1 equal 0}
  D^{(i)}_{r,v}(u_k(t,x))=0\qquad\text{when }r>t.
\end{equation}
\end{rem}

We conclude this section by the following useful lemma due to Morien \cite[Lemma 4.2]{Morien} for $d=1$
\begin{lem}\label{Morien}
  Assume \normalfont{\textbf{A1}}. For all $q\geq 1$, $T>0$ there exists $C>0$ such that for all $0<\varepsilon\leq s \leq t\leq T$ and $ 0\leq y\leq 1$,
  $$\sum_{i=1}^{d}\E\left[\left(\int_{s-\varepsilon}^{s}dr\int_{0}^{1}dv\left|\sum_{k=1}^{d}\left|D^{(i)}_{r,v}(u_k(t,y))\right|\right|^2\right)^q\right]\leq C \varepsilon^{q/2}.$$
\end{lem}
\section{Some new Malliavin calculus tools}\label{Malliavin calculus for adapted processes}
The main application of Malliavin calculus is the study of existence and
smoothness of densities for the probability laws. For the purposes of the proof
of our results, we introduce some new tools of Malliavin calculus that are
interesting in the frame of general adapted processes. Let us first state the
following definitions

\begin{defn}\label{defnotation}
 For a fixed integer $n\geq 1$,  let  $F=(F_1,\cdots,F_n)$ be an $\R^{n\times d}$-valued random vector such that for $i=1,\cdots,n$, $F_i=(F_{i,1},\cdots,F_{i,d})$ where $F_{i,k}\in \mathbb{D}^{1,p}$ for any $1\leq i\leq n$ and $1\leq k\leq d$. Let $\pi_n=(t_1,\cdots,t_n)$ with $0=t_0<t_1<\cdots<t_n\leq T$, we define the following matrices, for every $1\leq i,j\leq n$,
  $$\Gamma^{i,j}=\left(\Gamma^{i,j}_{k,l}\right)_{1\leq k,l\leq d}\quad \text{where}\quad \Gamma^{i,j}_{k,l}=\left<DF_{i,k}\,,\,DF_{j,l}\right>_{\HH_{t_{i-1},t_i}},$$
   here $\HH_{t_{i-1},t_i}=L^2([t_{i-1},t_i]\times[0,1],\R^d)$. We write $\Gamma_{F\,,\,\pi_n}$ for the $\pi_n$-Malliavin matrix of $F$. That is, the following block matrix
  $$\Gamma_{F\,,\,\pi_n}=\left(\Gamma^{i,j}\right)_{1\leq i,j\leq n}.$$
\end{defn}
\begin{rem}
For an $\R^d$-valued adapted process $X=(X(t))_{t\in [0,T]}$ which is smooth in the Malliavin sense, we have the following known property:
$$D^{(i)}_rX_k(t)=0\qquad\text{ for any }r>t.$$
Hence, the $\pi_n$-Malliavin matrix of the vector of the increments of $X$ is a triangular block matrix. So, the calculus of its determinant is obvious, which explains the practical usefulness of considering the $\pi_n$-Malliavin matrices for adapted processes.
\end{rem}
\begin{defn}
  With the same notations as in Definition \ref{defnotation},
  $F$ is said to be $\pi_n$-nondegenerate, if it satisfies the following three conditions:
  \begin{description}
    \item[(i)] For all $i=1,\cdots,n,$ and $k=1,\cdots,d,$ $F_{i,k}\in \mathbb{D}^{\infty}$.
    \item[(ii)] $\Gamma_{F\,,\,\pi_n}$ is invertible a.s. and we denote by $\Gamma^{-1}_{F\,,\,\pi_n}$ its inverse.
    \item[(iii)] $\left(\det{\Gamma_{F\,,\,\pi_n}}\right)^{-1}\in L^p$ for all $p\geq 1$.
  \end{description}
\end{defn}

The following integration by parts formula plays a crucial role in our paper.

Let $m\geq 1$. For any multi-index $\beta=(\beta_{1},\cdots,\beta_{m})$ with $\beta_{\theta}=(i_{\theta},k_{\theta})\in \{1,\cdots,n\}\times\{1,\cdots,d\}$, for $\theta=1,\cdots,m$, we introduce the following notations:
  $$\partial_{\beta_{\theta}}:=\frac{\partial}{\partial x_{i_{\theta},k_{\theta}}}\quad \text{for}\quad \theta=1,\cdots,m ,\quad\text{and}\quad \partial_{\beta}:=\partial_{\beta_1}\cdots\partial_{\beta_m}.$$
\begin{prop}\label{prop integration by parts}
  For a fixed integer $n\geq 1$, let $\pi_n=(t_1,\cdots,t_n)$ with $0=t_0<t_1<\cdots<t_n\leq T$ and  $F=(F_1,\cdots,F_n)$ be a ${\pi_n}$-nondegenerate
   random vector with values in $\R^{n\times d}$ such that $F_i=(F_{i,1},\cdots,F_{i,d})$, for $i=1,\cdots,n$. Let $G\in \mathbb{D}^{\infty}$ and let $g(x)\in C^{\infty}_{P}(\R^{n\times d})$, where $x=(x_{i,k}\,,\;1\leq i\leq n\,,\; 1\leq k\leq d)$.
  Then for all $m\geq 1$ and any multi-index $\beta=(\beta_{1},\cdots,\beta_{m})$, there exists $H^{\beta}_{\pi_n}(F\,,\,G)\in \mathbb{D}^{\infty}$ such that
  \begin{equation}\label{integration by parts}
    \E\left[(\partial_{\beta}g)(F)G\right]=\E\left[g(F) H^{\beta}_{\pi_n}(F\,,\,G)\right],
  \end{equation}
  where the random variables $H^{\beta}_{\pi_n}(F, G)$ are recursively given by
  \begin{equation}\label{H i k}
    H^{(i,k)}_{\pi_n}(F\,,\,G)=\sum_{j=1}^{n}\sum_{l=1}^{d}\delta\left(G\left(\Gamma^{-1}_{F\,,\,\pi_n}\right)^{i,j}_{k,l}DF_{j,l}^{}\mathbbm{1}_{[t_{j-1}\,,\,t_j]\times[0\,,\,1]}\right),
  \end{equation}
  \begin{equation}\label{H alpha}
    H^{\beta}_{\pi_n}(F\,,\,G)=H^{\beta_m}_{\pi_n}\left(F\,,\,H^{(\beta_1,\cdots,\beta_{m-1})}_{\pi_n}(F\,,\,G)\right).
  \end{equation}
  \end{prop}
  \begin{proof}
    By the chain rule \cite[Proposition 1.2.3]{Nualart} we have
    $$Dg(F)=\sum_{j=1}^{n}\sum_{l=1}^{d}\frac{\partial g}{\partial x_{j,l}}(F)DF_{j,l}.$$
    Therefore, for every $i=1,\cdots,n,$ and $k=1,\cdots,d,$
    \begin{align*}
      \left<DF_{i,k}\,,\,Dg(F)\right>_{\HH_{t_{i-1},t_i}} &= \sum_{j=1}^{n}\sum_{l=1}^{d}\frac{\partial g}{\partial x_{j,l}}(F)\left<DF_{i,k}\,,\,DF_{j,l}\right>_{\HH_{t_{i-1},t_i}}\\
       &= \sum_{j=1}^{n}\sum_{l=1}^{d}\frac{\partial g}{\partial x_{j,l}}(F)\,\Gamma^{i,j}_{k,l}.
    \end{align*}
    Hence, for any $i=1,\cdots,n,$ and $k=1,\cdots,d,$
    $$\frac{\partial g}{\partial x_{i,k}}(F)= \sum_{j=1}^{n}\sum_{l=1}^{d}\left<DF_{j,l}\,,\,Dg(F)\right>_{\HH_{t_{j-1},t_j}}\left(\Gamma^{-1}_{F\,,\,\pi_n}\right)^{i,j}_{k,l}.$$
    And, consequently, we obtain
    \begin{equation}\label{G partial}
      G\frac{\partial g}{\partial x_{i,k}}(F)= \sum_{j=1}^{n}\sum_{l=1}^{d}G\left<DF_{j,l}\,,\,Dg(F)\right>_{\HH_{t_{j-1},t_j}}\left(\Gamma^{-1}_{F\,,\,\pi_n}\right)^{i,j}_{k,l}.
    \end{equation}
    On the other hand, we have for every $j=1,\cdots,n,$ and $l=1,\cdots,d,$
  \begin{equation}\label{Ht to H}
    \left<DF_{j,l}\,,\,Dg(F)\right>_{\HH_{t_{j-1},t_j}}=\left<DF_{j,l}\,\mathbbm{1}_{[t_{j-1}\,,\,t_j]\times[0\,,\,1]}\,,\,Dg(F)\right>_{\HH}.
  \end{equation}
    Finally, taking expectations in \eqref{G partial}  and using \eqref{Ht to H} and the duality formula \eqref{duality} we get
    $$\E[\partial_{(i,k)} g(F)G]=\E[g(F)H^{(i,k)}_{\pi_n}(F\,,\,G)],$$
    where $H^{(i,k)}_{\pi_n}(F\,,\,G)$ is given by \eqref{H i k}. The equation \eqref{H alpha} follows by recurrence.
  \qed\end{proof}
  The below lemma will be devoted to get some estimations of the
$\|\cdot\|_{k,p}$-norm of elements of the Malliavin matrix
  \begin{lem}\label{inner product leq norms}
    Let $0\leq s<t\leq T$ and $F,G\in \mathbb{D}^{\infty}$, then we have the following
    \begin{equation}\label{norm DF DG }
     \|\left<DF\,,\,DG\right>_{\HH_{s,t}}\|_{k,p}\leq C\|DF\|_{k,2p,\HH} \,\|DG\|_{k,2p,\HH}.
    \end{equation}
  \end{lem}
  \begin{proof}
    By definition, we have
    \begin{equation}\label{norm DF DG}
      \begin{split}
         \|\left<DF\,,\,DG\right>_{\HH_{s,t}}\|_{k,p} =& \left\{\vphantom{\frac{1}{2}}\E[|\left<DF\,,\,DG\right>_{\HH_{s,t}}|^p]\right.\\
           &\left.+\sum_{j=1}^{k}\E[\|D^j\left<DF\,,\,DG\right>_{\HH_{s,t}}\|_{\HH^{\otimes j}}^p]\right\}^{\frac{1}{p}}.
      \end{split}
    \end{equation}
  We can easily check that
  \begin{equation}\label{inner product DF DG Hts leq}
    \E[|\left<DF\,,\,DG\right>_{\HH_{s,t}}|^p]\leq \|DF\|^p_{k,2p,\HH}\,\|DG\|^p_{k,2p,\HH}.
  \end{equation}
  On the other hand, we get for $j\geq 1$ and $p>2,$
  \begin{align}\label{inner product DjF DjG Hts leq}
    &E[\|D^j\left<DF\,,\,DG\right>_{\HH_{s,t}}\|_{\HH^{\otimes j}}^p] \nonumber\\
    &= E\left[\left\|D^j\left(\int_{s}^{t}\int_{0}^{1} D_{r,x}F\cdot D_{r,x}G\,drdx\right)\right\|_{\HH^{\otimes j}}^p\right]  \nonumber\\
    &=\E\left[\left\{\sum_{i_1,\cdots,i_j=1}^{d}\int_{0}^{T}dr_1\int_{0}^{1}dx_1\cdots\int_{0}^{T}dr_j\int_{0}^{1}dx_j\right.\right.\nonumber\\
    &\qquad\qquad\qquad\qquad\left.\left.\left|\sum_{l=1}^{d}\int_{s}^{t}\int_{0}^{1}D^{(i_1)}_{r_1,x_1}\cdots D^{(i_j)}_{r_j,x_j}(D^{(l)}_{r,x} F\cdot D^{(l)}_{r,x}G)\,drdx\right|^2\right\}^{\frac{p}{2}}\right].
    \end{align}
    Let $J=\{i_1,\cdots,i_j\}$ and $I\subset \{i_1,\cdots,i_j\}$ such that $I=\{i_{k_1},\cdots,i_{k_m}\}$, and $\alpha_I=(t_{k_1},x_{k_1},\cdots,t_{k_m},x_{k_m})$, we put $D^I_{\alpha_I}F:=D^{(i_{k_1})}_{t_{k_1},x_{k_1}}\cdots D^{(i_{k_m})}_{t_{k_m},x_{k_m}}F$. When $I=\varnothing$, we set $D^{I}F=F$. Then we have for $F,G\in \mathbb{D}^{\infty}$ and $\alpha=(t_1,x_1,\cdots,t_j,x_j)$, the following Leibniz's rule
    \begin{equation}\label{Leibnizs rule}
      D_{\alpha}^J(F\cdot G)=\sum_{I\subset \{i_1,\cdots,i_j\}}D^I_{\alpha_I}F\cdot D^{I^c}_{\alpha_{I^c}}G,
    \end{equation}
    where $I^c$ is the complement of $I$. We denote by $|I|$ the cardinality of $I$. Combining  \eqref{inner product DjF DjG Hts leq} and \eqref{Leibnizs rule}, we get
    \begin{align*}
    &E[\|D^j\left<DF\,,\,DG\right>_{\HH_{s,t}}\|_{\HH^{\otimes j}}^p] \nonumber\\
    &=\E\left[\left\{\sum_{i_1,\cdots,i_j=1}^{d}\int_{0}^{T}dr_1\int_{0}^{1}dx_1\cdots\int_{0}^{T}dr_j\int_{0}^{1}dx_j\right.\right.\nonumber\\
    &\qquad\qquad\qquad\left.\left.\left|\sum_{l=1}^{d}\sum_{I\subset \{i_1,\cdots,i_j\}}\int_{s}^{t}\int_{0}^{1}D^I_{\alpha_I}D^{(l)}_{r,x} F \cdot D^{I^c}_{\alpha_{I^c}}D^{(l)}_{r,x} G \,drdx\vphantom{\sum_{l=1}^{d}}\right|^2\right\}^{\frac{p}{2}}\vphantom{\sum_{l=1}^{d}}\right].\nonumber
    \end{align*}
    Therefore, by the convexity, H\"{o}lder's inequality, and Fubini's theorem, we obtain that this last term is less than or equal to
    \begin{align}\label{DFDG sum}
    &\hat{C} \sum_{i_1,\cdots,i_j=1}^{d}\sum_{l=1}^{d}\sum_{I\subset \{i_1,\cdots,i_j\}}  \E\left[\left\{\int_{([0,T]\times[0,1])^{|I|}}d\alpha_I\int_{0}^{T}dr\int_{0}^{1}dx\,\left|D^I_{\alpha_I}D^{(l)}_{r,x} F\right|^2\right\}^{p}\right]^{\frac{1}{2}}\nonumber\\
       &\qquad\qquad\qquad\cdot \E\left[\left\{\int_{([0,T]\times[0,1])^{|I^c|}}d\alpha_{I^c}\int_{0}^{T}dr\int_{0}^{1}dx\, \left|D^{I^c}_{\alpha_{I^c}}D^{(l)}_{r,x} G\right|^2 \vphantom{\sum_{l=1}^{d}}\right\}^{p}\vphantom{\sum_{l=1}^{d}}\right]^{\frac{1}{2}}\nonumber\\
    & \leq\hat{C} \sum_{i_1,\cdots,i_j=1}^{d}\sum_{l=1}^{d}\sum_{I\subset \{i_1,\cdots,i_j\}} \E\left[\left\|D^{|I|+1}F\right\|_{\HH^{\otimes (|I|+1)}}^{2p}\right]^{\frac{1}{2}}\cdot \E\left[\left\|D^{|I^c|+1}G\right\|_{\HH^{\otimes (|I^c|+1)}}^{2p}\right]^{\frac{1}{2}}\nonumber\\
    &\leq \hat{C} \sum_{i_1,\cdots,i_j=1}^{d}\sum_{l=1}^{d}\sum_{I\subset \{i_1,\cdots,i_j\}}\left\|DF\right\|^p_{k,2p,\HH}\left\|DG\right\|^p_{k,2p,\HH} \leq C\,\left\|DF\right\|^p_{k,2p,\HH}\left\|DG\right\|^p_{k,2p,\HH}.
    \end{align}
One can easily derive from \eqref{norm DF DG}, \eqref{inner product DF DG Hts leq}, and \eqref{DFDG sum} the inequality \eqref{norm DF DG }.
  \qed\end{proof}
The next lemma gives a sharp estimate of the $\|\cdot\|_{0,2}$-norm of the random variables $H^{\beta}_{\pi_n}(F\,,\,G)$
\begin{lem}\label{H estimation}
  For a fixed integer $n\geq 1$, let $\pi_n=(t_1,\cdots,t_n)$ with $0=t_0<t_1<\cdots<t_n\leq T$ and  $F=(F_1,\cdots,F_n)$ be a $\pi_n$-nondegenerate
   random vector with values in $\R^{n\times d}$ such that, $F_i=(F_{i,1},\cdots,F_{i,d})$, for $i=1,\cdots,n$. Let  $\beta=(\beta_1,\cdots,\beta_m)$ with $\beta_{\theta}=(i_{\theta},k_{\theta})\in \{1,\cdots,n\}\times\{1,\cdots,d\}$, for $\theta=1,\cdots, m$, then there exists a constant $C>0$ such that
  \begin{equation}\label{H t1 tn inequality}
    \begin{split}
       &\left\|H^{\beta}_{\pi_n}(F\,,\,1)\right\|_{0,2} \leq C\,\left\|\left(\det{\Gamma_{F\,,\,\pi_n}}\right)^{-1}\right\|^m_{m,2^{m+2}}  \\
         &\qquad\qquad\cdot \prod_{\theta=1}^{m}\left\|DF_{i_{\theta},k_{\theta}}\right\|_{m,2^{2(m+nd)},\HH}\prod_{(i_0,k_0)\in O_{(i_{\theta},k_{\theta})}}\left\|DF_{i_0,k_0}\right\|^2_{m,2^{2(m+nd)},\HH},
    \end{split}
  \end{equation}
  where $O_{(i_{\theta},k_{\theta})}=\left\{(i_0,k_0)\in \{1,\cdots,n\}\times \{1,\cdots,d\}\,;\;(i_0,k_0)\neq (i_{\theta},k_{\theta})\right\}.$
\end{lem}

\begin{proof}
 By \eqref{H alpha} and \eqref{H i k}, we write
 \begin{align*}
    &\left\|H^{\beta}_{\pi_n}(F\,,\,1)\right\|_{0,2}\\
   &= \left\|H^{\beta_m}_{\pi_n}\left(F\,,\,H^{(\beta_1,\cdots,\beta_{m-1})}_{\pi_n}(F\,,\,1)\right)\right\|_{0,2}\nonumber \nonumber\\
    &=\left\|\sum_{j=1}^{n}\sum_{l=1}^{d}\delta\left(H^{(\beta_1,\cdots,\beta_{m-1})}_{\pi_n}(F\,,\,1)\left(\Gamma^{-1}_{F\,,\,\pi_n}\right)^{i_m,j}_{k_m,l}DF_{j,l}\mathbbm{1}_{[t_{j-1}\,,\,t_j]\times[0\,,\,1]}\right)\right\|_{0,2}.
    \end{align*}
     According to \eqref{continuite of the divergence operator}, and H\"{o}lder's inequality for the Malliavin norms (cf. \cite[Proposition 1.10]{Watanabe}), we obtain that this last term is less than or equal to
    \begin{align}\label{H beta estimation}
    & C\,\left\|H^{(\beta_1,\cdots,\beta_{m-1})}_{\pi_n}(F\,,\,1)\right\|_{1,2^2}\,\sum_{j=1}^{n}\sum_{l=1}^{d} \left\|\left(\Gamma^{-1}_{F\,,\,\pi_n}\right)^{i_m,j}_{k_m,l}\right\|_{1,2^3} \left\|DF_{j,l}\mathbbm{1}_{[t_{j-1},t_j]\times[0,1]}\right\|_{1,2^3,\HH}\nonumber\\
    & \leq C\,\left\|H^{(\beta_1,\cdots,\beta_{m-1})}_{\pi_n}(F\,,\,1)\right\|_{1,2^2}\,\sum_{j=1}^{n}\sum_{l=1}^{d} \left\|\left(\Gamma^{-1}_{F\,,\,\pi_n}\right)^{i_m,j}_{k_m,l}\right\|_{1,2^3} \left\|DF_{j,l}\right\|_{1,2^3,\HH}.
 \end{align}
 On the other hand, we know that the inverse of the matrix $\Gamma_{F\,,\,\pi_n}$ is the transpose of its cofactor matrix, that we denote by $A_{F\,,\,\pi_n}$,
multiplied by the inverse of its determinant i.e.,
 \begin{equation}\label{Gamma determinant A tronspose}
   \Gamma^{-1}_{F\,,\,\pi_n}=\frac{1}{\det{\Gamma_{F\,,\,\pi_n}}} A'_{F\,,\,\pi_n},
 \end{equation}
  where $A'_{F\,,\,\pi_n}$ is the transpose of $A_{F\,,\,\pi_n}$. Set $N=\{1,\cdots,n\}$ and $D=\{1,\cdots,d\}$. Let $B(i,k\,;\,j,l)=\left(b(i_0,k_0\,;\,j_0,l_0)\right),$ where $(i_0,k_0),(j_0,l_0)\in O:=\{(p,q)\in N\times D\,;\;(p,q)\neq (n,d)\}$, be the $(n\times d-1)\times(n\times d-1)$-matrix obtained by removing from $\Gamma_{F\,,\,\pi_n}$ its row $(i,k)$  and column $(j,l)$. Let $O_{(i,k)}$ be the set of $(i_0,k_0)\in N\times D$ such that $(i_0,k_0)\neq (i,k)$, $O_{(j,l)}$ the set of $(j_0,l_0)\in N\times D$ with $(j_0,l_0)\neq (j,l)$,
  $\Xi=\{\eta\,;\;\eta \text{ permutation of}\,\, O\}$, and
  $\Pi:=\{\varrho\,;\;\varrho:O_{(i,k)}\to O_{(j,l)}\;\text{bijective map}\}$,  hence by H\"{o}lder’s inequality for the Malliavin norms and Lemma \ref{inner product leq norms} we have
  \begin{align}\label{A cofactor elelments}
    \left\|(A_{F\,,\,\pi_n})^{i,j}_{k,l}\right\|_{1,2^4} &= \left\|\det\left(B(i,k\,;\,j,l)\right)\right\|_{1,2^4}\nonumber\\
    &=\left\|\sum_{\eta\in \Xi}\,\varepsilon(\eta)\prod_{(p,q)\in O}b\left(p,q\,;\,\eta(p,q)\right)\right\|_{1,2^4}  \nonumber\\
    &\leq \sum_{\eta\in \Xi}\,\prod_{(p,q)\in O}\left\|b\left(p,q\,;\,\eta(p,q)\right)\right\|_{1,2^{nd +2}} \nonumber\\
    &= \sum_{\varrho\in \Pi}\, \prod_{(i_0,k_0)\in O_{(i,k)}}\left\|\left<DF_{i_0,k_0}\,,\,DF_{\varrho(i_0,k_0)} \right>_{\HH_{t_{i_0-1},t_{i_0}}}\right\|_{1,2^{nd +2}}\nonumber\\
    &\leq C \sum_{\varrho\in \Pi}\, \prod_{(i_0,k_0)\in O_{(i,k)}}\left\|DF_{i_0,k_0}\right\|_{1,2^{2(nd) +4},\HH}\left\|DF_{\varrho(i_0,k_0)}\right\|_{1,2^{2(nd) +4},\HH}.
  \end{align}
  According to \eqref{Gamma determinant A tronspose}, \eqref{A cofactor elelments}, and H\"{o}lder's inequality for the Malliavin norms, we get
  \begin{align}\label{Gamma im j km l}
    \left\|\left(\Gamma^{-1}_{F\,,\,\pi_n}\right)^{i_m,j}_{k_m,l}\right\|_{1,2^3}
    &\leq C\,\left\|\left(\det{\Gamma_{F\,,\,\pi_n}}\right)^{-1}\right\|_{1,2^4}\nonumber\\
    &\cdot\sum_{\mu\in \Pi_m}\, \prod_{(j_0,l_0)\in O_{(j,l)}}\left\|DF_{j_0,l_0}\right\|_{1,2^{2nd +4},\HH}\left\|DF_{\mu(j_0,l_0)}\right\|_{1,2^{2nd +4},\HH},
  \end{align}
  where $\Pi_m:=\{\mu\,;\;\mu:O_{(j,l)}\to O_{(i_m,k_m)}\;\text{bijective map}\}$ and $O_{(i_m,k_m)}$ is the set of $(i_0,k_0)\in N\times D$ such that $(i_0,k_0)\neq (i_m,k_m)$. We derive from  \eqref{Gamma im j km l} that
  \begin{align}\label{product 2}
     & \left\|\left(\Gamma^{-1}_{F\,,\,\pi_n}\right)^{i_m,j}_{k_m,l}\right\|_{1,2^3}\nonumber \\
     &\leq C\,\left\|\left(\det{\Gamma_{F\,,\,\pi_n}}\right)^{-1}\right\|_{1,2^4}\sum_{\mu\in \Pi_m}\, \prod_{(j_0,l_0)\in O_{(j,l)}}\left\|DF_{j_0,l_0}\right\|_{1,2^{2nd +4},\HH}\nonumber \\
     &\qquad\qquad\qquad\qquad\cdot\prod_{(i_0,k_0)\in O_{(i_m,k_m)}}\left\|DF_{i_0,k_0}\right\|_{1,2^{2nd +4},\HH} \nonumber\\
     &= C\,|\Pi_m|\,\left\|\left(\det{\Gamma_{F\,,\,\pi_n}}\right)^{-1}\right\|_{1,2^4}\left\|DF_{j,l}\right\|_{1,2^{2nd +4},\HH}\left\|DF_{i_m,k_m}\right\|_{1,2^{2nd +4},\HH}\nonumber \\
     &\qquad\qquad\qquad\qquad\cdot\prod_{(i_0,k_0)\in O_{(j,l)}\cap\, O_{(i_m,k_m)}}\left\|DF_{i_0,k_0}\right\|^2_{1,2^{2nd +4},\HH},
  \end{align}
  where $|\Pi_m|$ is the cardinality of $\Pi_m$. Combining \eqref{H beta estimation} and \eqref{product 2}, we obtain
  \begin{align*}
     & \left\|H^{\beta}_{\pi_n}(F\,,\,1)\right\|_{0,2} \nonumber  \\
     & \leq C_1\,\left\|H^{(\beta_1,\cdots,\beta_{m-1})}_{\pi_n}(F\,,\,1)\right\|_{1,2^2}\left\|\left(\det{\Gamma_{F\,,\,\pi_n}}\right)^{-1}\right\|_{1,2^4}\left\|DF_{i_m,k_m}\right\|_{1,2^{2nd +4},\HH}\nonumber\\
     &\cdot \sum_{j=1}^{n}\sum_{l=1}^{d}\left\|DF_{j,l}\right\|_{1,2^{2nd +4},\HH}\prod_{(i_0,k_0)\in O_{(j,l)}\cap\, O_{(i_m,k_m)}}\left\|DF_{i_0,k_0}\right\|^2_{1,2^{2nd +4},\HH}  \left\|DF_{j,l}\right\|_{1,2^3,\HH}\nonumber \\
     & \leq C_2\,\left\|H^{(\beta_1,\cdots,\beta_{m-1})}_{\pi_n}(F\,,\,1)\right\|_{1,2^2}\left\|\left(\det{\Gamma_{F\,,\,\pi_n}}\right)^{-1}\right\|_{1,2^4}\left\|DF_{i_m,k_m}\right\|_{1,2^{2nd +4},\HH}\nonumber\\
     & \qquad\qquad\qquad\qquad\cdot \prod_{(i_0,k_0)\in O_{(i_m,k_m)}}\left\|DF_{i_0,k_0}\right\|^2_{1,2^{2nd +4},\HH}
  \end{align*}
  Finally, by recurrence on $m$ we get the inequality \eqref{H t1 tn inequality}, which finishes  the proof of Lemma \ref{H estimation}.
\qed\end{proof}

Now we will state the criterion for smoothness of density for a random vector which is $\pi_n$-nondegenerate, and give the formula of its derivatives. The proof is similar to \cite[Theorem 2.1.4]{Nualart}.
\begin{thm}\label{existence of density general}
For a fixed integer $n\geq 1$, let $\pi_n=(t_1,\cdots,t_n)$ with  $0=t_0<t_1<\cdots<t_n\leq T$ and  $F=(F_1,\cdots,F_n)$ be a $\pi_n$-nondegenerate random vector with values in $\R^{n\times d}$  such that $F_i=(F_{i,1},\cdots,F_{i,d})$, for $i=1,\cdots,n$. Then $F$
possesses a density $p_{\pi_n}(x)$, where $x=(x_{i,k}\,;\;1\leq i\leq n\,,\;1\leq k\leq d)\in \R^{n\times d}$, which is infinitely differentiable and given by
\begin{equation}\label{density}
  p_{\pi_n}(x)=\E\left[\mathbbm{1}_{\{F>x\}}H^{\gamma}_{\pi_n}(F\,,\,1)\right],
\end{equation}
where $\gamma=\left((i,k)\,;\;1\leq i\leq n\,,\;1\leq k\leq d\right)$. Fix $m\geq 1$. For any multi-index $\beta=(\beta_{1},\cdots,\beta_{m})$ with $\beta_{\theta}=(i_{\theta},k_{\theta})\in \{1,\cdots,n\}\times\{1,\cdots,d\}$, for $\theta=1,\cdots,m$, we have
\begin{equation}\label{derivatives of density}
  \partial_{\beta}p_{\pi_n}(x)=(-1)^{m}\E\left[\mathbbm{1}_{\{F>x\}}H^{(\beta,\gamma)}_{\pi_n}(F\,,\,1)\right],
\end{equation}
where $\mathbbm{1}_{\{F>x\}}:=\prod_{i=1}^{n}\prod_{k=1}^{d}\mathbbm{1}_{\{F_{i,k}>x_{i,k}\}}$.
\end{thm}
\section{Proof of Theorem \ref{estimation density}}
Our purpose in  this section is to establish the Gaussian-type lower bound for the density
of $u(t,x)-u(s,x)$, when $0\leq s<t \leq T$, and the upper bound of Gaussian-type for the partial derivatives of the density
of $(u(t_1,x)-u(t_0,x),\cdots,u(t_n,x)-u(t_{n-1},x))$, where $0=t_0<t_1<\cdots<t_n\leq T$.
\subsection{The Gaussian-type lower bound}
The proof of Theorem \ref{estimation density}\textbf{(a)} is quite similar to that in \cite[Section 5]{DalangKhoshnevisanNualartmultiplicative} (see also \cite[Theorem 10]{Higa} for the original work in dimension $d=1$). Thus we will only sketch the main ideas.

\begin{rem}
    For clarity reasons, we borrow most of the notations in this subsection  from \cite[Section 5]{DalangKhoshnevisanNualartmultiplicative} and \cite{Higa}. Therefore, the notations used here are independent of the rest of this paper.
\end{rem}

\begin{proof}[Proof of Theorem \ref{estimation density}\normalfont{\textbf{(a)}}] Let $x\in (0,1)$ and $0\leq s<t \leq T$, the proof of Theorem \ref{estimation density}\normalfont{\textbf{(a)}} follows the same lines as in \cite[Section 5]{DalangKhoshnevisanNualartmultiplicative} (or \cite{Higa} for $d=1$).
We only sketch the main points where there is a difference between the chose of $F=u(t,x)$ (the study of \cite{DalangKhoshnevisanNualartmultiplicative}) and $F=u(t,x)-u(s,x)$. The idea of Kohatsu-Higa \cite{Higa} is to show that  $u(t,x)-u(s,x)$ is  a $d$-dimensional uniformly elliptic random vector and therefore apply \cite[Theorem 5]{Higa}.

Set $g(r,v):=G_{t-r}(x,v)$. Let us consider a sufficiently fine partition $\{s=t_0<\cdots<t_N=t\}$. By the properties of $G$, i.e., symmetry, semi-group property and, Gaussian-type bound, there exist two positive constants $c_1$ and $c_2$ such that
$$c_1|t-s|^{1/4}\leq \|g\|_{L^2([s,t]\times[0,1])}\leq c_2|t-s|^{1/4}.$$
Let
$$\tilde{F}^i_n=F^i_n-u_i(s,x),$$
where for $1\leq i\leq d$ and $0\leq n\leq N$, $F^i_n$ are given as in  \cite[Section 5]{DalangKhoshnevisanNualartmultiplicative} by
\begin{equation*}
  \begin{split}
     F^i_n &= \int_{0}^{t_n}\int_{0}^{1}G_{t-r}(x,v)\sum_{j=1}^{d}\sigma_{ij}(u(r,v))W^j(dr,dv) \\
       & \qquad+\int_{0}^{t_n}\int_{0}^{1}G_{t-r}(x,v)b_i(u(r,v))dv\,dr.
  \end{split}
\end{equation*}
Point out that $\tilde{F}^i_n\in \mathcal{F}_{t_n}$. We will need the following lemma.
\begin{lem}\label{lem lower}
  We assume \normalfont{\textbf{A1}} and \normalfont{\textbf{A2}}. Then
  \begin{description}
    \item[(1)] $\|\tilde{F}^i_n\|_{k,p}\leq c_{k,p},\;\;1\leq i\leq d$;
    \item[(2)] $\|\left(\det \gamma_{\tilde{F}_n}(t_{n-1})\right)^{-1}\|_{p,t_{n-1}}\leq c_p (\Delta_{n-1}(g))^{-d}:= c_p(\|g\|^2_{L^2([t_{n-1},t_n]\times[0,1])})^{-d},$
  \end{description}
  where $\|\cdot\|_{p,t_{n-1}}$ denotes the conditional $L^p$-norm and $\gamma_{\tilde{F}_n}(t_{n-1})$ is the conditional Malliavin matrix of $\tilde{F}_n$ given $\mathcal{F}_{t_n}$.
\end{lem}
\begin{proof}[Proof of Lemma \ref{lem lower}]
  The point \normalfont{\textbf{(1)}} is a consequence of \cite[Lemma 5.1(i)]{DalangKhoshnevisanNualartmultiplicative} and \eqref{sup derivative SHE}. Otherwise, by the fact that $s\leq t_{n-1}$ and \eqref{derivative order 1 equal 0}, we have the following:
  $$\gamma_{\tilde{F}_n}(t_{n-1})=\gamma_{F_n}(t_{n-1}),$$
  where $\gamma_{F_n}(t_{n-1})$ is the conditional Malliavin matrix of $F_n$ given $\mathcal{F}_{t_n}$. Then we can conclude the proof of \normalfont{\textbf{(2)}} by \cite[Lemma 5.1(ii)]{DalangKhoshnevisanNualartmultiplicative} or \cite[Lemma 7]{Higa} (the last reference is for $d=1$, but the same ideas, in the proof, still work for $d>1$).
\qed\end{proof}
Continuing the proof of Theorem \ref{estimation density}\normalfont{\textbf{(a)}}. We remark that
$$\tilde{F}^i_n-\tilde{F}^i_{n-1}=F^i_n-F^i_{n-1}.$$
Hence, in order to get the expansion of $\tilde{F}^i_n-\tilde{F}^i_{n-1}$ as in \cite[Lemma 9]{Higa}, one has to obtain that expansion (i.e., as in \cite[Lemma 9]{Higa}) for $F^i_n-F^i_{n-1}$. The remainder of the proof is the same as in \cite[Section 5]{DalangKhoshnevisanNualartmultiplicative}.
\qed\end{proof}
\subsection{The Gaussian-type upper bound for the partial derivatives of the density}
Our aim in this subsection is to prove Theorem  \ref{estimation density}\textbf{(b)}. Let $p_{\pi_n,x}(\xi)$ be the joint density of the $\R^{n\times d}$-valued random vector
\begin{equation}\label{Z}
  Z=(u(t_1,x)-u(t_0,x),\cdots,u(t_n,x)-u(t_{n-1},x)),
\end{equation}
where $u(t_i,x)-u(t_{i-1},x)=(u_1(t_i,x)-u_1(t_{i-1},x),\cdots,u_d(t_i,x)-u_d(t_{i-1},x))$,  $\xi=(\xi_{i,k}\,;\; 1\leq i\leq n\,,\;1\leq k\leq d)\in \R^{n\times d}$,  $x\in (0,1)$, and $\pi_n=(t_1,\cdots,t_n)$ with $0=t_0<t_1<\cdots<t_n\leq T$. The existence of this joint density, which is infinitely differentiable, will be a consequence of Theorem \ref{existence of density general}, Proposition \ref{prop derivatives SHE}, and Theorem \ref{thm estimation det moins 1}.

The following proposition gives an upper bound for the Malliavin norm of the derivative of the increments of the process $\{u(t,x)\,,\;t\in [0,T]\}$.
\begin{prop}[Proposition 6.2 in \cite{DalangKhoshnevisanNualartmultiplicative}]\label{Malliavin norm increment SHE}
  Assume \normalfont{\textbf{A1}}. Then for any $0\leq s \leq t\leq T$, $x\in [0,1]$, $p>1$, and $m\geq 1$,
  $$\E\left[\|D^m(u_k(t,x)-u_k(s,x))\|^p_{\mathcal{H}^{\otimes m}}\right]\leq C_T \,|t-s|^{p/4}\,,\qquad k=1,\cdots,d.$$
\end{prop}

Now we will investigate the $\pi_n$-Malliavin matrix, $\Gamma_{Z\,,\,\pi_n}$, of $Z$ ($Z$ is given by \eqref{Z}).  Note that $\Gamma_{Z\,,\,\pi_n}=\left(\Gamma^{i,j}\right)_{1\leq i,j\leq n}$ is the random block matrix, where $\Gamma^{i,j}=\left(\Gamma^{i,j}_{k,l}\right)_{1\leq k,l\leq d}$ and the $\Gamma^{i,j}_{k,l}$ are given by
  \begin{equation}\label{Gamma i j k l}
    \Gamma^{i,j}_{k,l}=\left<D(u_k(t_i,x)-u_k(t_{i-1},x))\,,\,D(u_l(t_j,x)-u_l(t_{j-1},x))\right>_{\HH_{t_{i-1},t_i}},
  \end{equation}
  here $1\leq i,j\leq n$, $1\leq k,l\leq d$, and $\HH_{t_{i-1},t_i}=L^2\left([t_{i-1},t_i]\times[0,1],\R^d\right)$. The matrix $\Gamma_{Z\,,\,\pi_n}$ is  not  a symmetric matrix, in general, (but the matrices $\Gamma^{i,i}$, for $i=1,\cdots,n$, are symmetric). Based on the formula \eqref{derivative order 1 equal 0}, we get the following key remark
  \begin{rem}
   $\Gamma_{Z\,,\,\pi_n}$ is a triangular block matrix almost surly, i.e., for all $1\leq i,j\leq n$ with $j<i$, we have $\Gamma^{i,j}\equiv 0$ a.s.
  \end{rem}
A consequence of the above remark is that
\begin{equation}\label{det malliavin matrix}
  \det\left(\Gamma_{Z\,,\,\pi_n}\right)=\prod_{i=1}^{n}\det\left(\Gamma^{i,i}\right)\qquad\text{a.s.}
\end{equation}

The below theorem gives an estimate on the Malliavin norm of the determinant of the inverse of the matrix $\Gamma_{Z\,,\,\pi_n}$.
\begin{thm}\label{thm estimation det moins 1}
  Assume \normalfont{\textbf{A1}} and \normalfont{\textbf{A2}}. Let $\pi_n=(t_1,\cdots,t_n)$ with $0=t_0<t_1<\cdots<t_n\leq T$, $x\in (0,1)$, and $Z$ given by \eqref{Z}, then for any $k\geq 0$, $p>1$,
  $$\|\left(\det \Gamma_{Z\,,\,\pi_n}\right)^{-1}\|_{k,p}\leq K\prod_{i=1}^{n}(t_i-t_{i-1})^{-d/2},$$
  where $K$ is a positive constant.
\end{thm}
\begin{proof}
  By \eqref{det malliavin matrix} and H\"{o}lder’s inequality for the Malliavin norms (cf. \cite[Proposition 1.10]{Watanabe}), we get
  $$\|\left(\det \Gamma_{Z\,,\,\pi_n}\right)^{-1}\|_{k,p}\leq \prod_{i=1}^{n}\|(\det \Gamma^{i,i})^{-1}\|_{k,2^{n-1}p}.$$
  Let $\tilde{p}=2^{n-1}p$. By definition, we have
  \begin{equation}\label{def norm Malliavn Gamma ii}
    \begin{split}
       \|(\det \Gamma^{i,i})^{-1}\|_{k,\tilde{p}} &= \left\{\vphantom{\frac{1}{2}}\E\left[\left|(\det \Gamma^{i,i})^{-1}\right|^{\tilde{p}}\right]\right.\\
       &\left.\qquad+\sum_{l=1}^{k}\E\left[\left\|D^l(\det \Gamma^{i,i})^{-1}\right\|^{\tilde{p}}_{\HH^{\otimes l}}\right]\right\}^{\frac{1}{\tilde{p}}}.
    \end{split}
  \end{equation}
  To estimate the moments of the inverse of the determinant of the matrix $\Gamma^{i,i}$, we use standard arguments. We follow \cite{MoretNualart}, Lemma 10, and the proof of (4.14) in \cite{DalangNualart}. We have the following lower bound for the determinant
  \begin{align}\label{minoration det}
    \det \Gamma^{i,i} &\geq \inf_{\|\xi\|=1}\left(\xi'\,\Gamma^{i,i} \,\xi\right)^d \nonumber\\
    &= \inf_{\|\xi\|=1}\left(\sum_{l=1}^{d}\int_{t_{i-1}}^{t_i}\int_{0}^{1}\left|\sum_{k=1}^{d}D^{(l)}_{r,v}\left(u_k(t_i,x)-u_k(t_{i-1},x)\right)\xi_k\right|^2 dv\,dr\right)^d.
  \end{align}
  Using \eqref{first derivative SHE}, \eqref{akirvtx}, and \eqref{derivative order 1 equal 0}, we get for all $x\in (0,1)$, $t_{i-1}<r<t_i$,
  \begin{equation}\label{derivative of u moins u}
       D^{(l)}_{r,v}\left(u_k(t_i,x)-u_k(t_{i-1},x)\right) = G_{t_i-r}(x,v)\sigma_{k,l}(u(r,v))+a_k(l,r,v,t_i,x),\\
  \end{equation}
  where $a_k(l,r,v,t_i,x)$ is given by \eqref{akirvtx}. According to \eqref{derivative of u moins u} and \normalfont{\textbf{A2}}, for any $h\in (0,1]$, we obtain that the expression in parentheses in \eqref{minoration det} is bounded below by
  \begin{align*}
    &\sum_{l=1}^{d}\int_{t_i-h(t_i-t_{i-1})}^{t_i}\int_{0}^{1}\left|\sum_{k=1}^{d}\xi_k\left(G_{t_i-r}(x,v)\sigma_{k,l}(u(r,v))+a_k(l,r,v,t_i,x)\right)\right|^2 dv\,dr \\
     &\geq \frac{\rho^2}{2}\int_{t_i-h(t_i-t_{i-1})}^{t_i}dr\int_{0}^{1}dv\,G^2_{t_i-r}(x,v)\\
     &\qquad\qquad\qquad\qquad\qquad-\sum_{l=1}^{d}\int_{t_i-h(t_i-t_{i-1})}^{t_i}dr\int_{0}^{1}dv\left|\sum_{k=1}^{d}\xi_k\, a_k(l,r,v,t_i,x)\right|^2\\
     &\geq c\,\frac{\rho^2}{2}\sqrt{h(t_i-t_{i-1})}-I_h,
  \end{align*}
where
\begin{equation}\label{Ih}
  I_h=\sup_{\|\xi\|=1}\sum_{l=1}^{d}\int_{t_i-h(t_i-t_{i-1})}^{t_i}dr\int_{0}^{1}dv\left|\sum_{k=1}^{d}\xi_k\, a_k(l,r,v,t_i,x)\right|^2.
\end{equation}

We choose $y$ such that $c\,\rho^2\sqrt{h(t_i-t_{i-1})}=4y^{-1/d}$, and point out that since $h\leq 1$, we have $y\geq a:= 4^d c^{-d}\rho^{-2d}(t_i-t_{i-1})^{-d/2}$. Furthermore, as $h$ varies in $(0,1]$, $y$ varies in $[a,\infty)$. By Chebyshev’s inequality, we get that for any $q\geq 1$,
\begin{align*}
  \mathbb{P}\left[\det \Gamma^{i,i}<\frac{1}{y}\right] &\leq  \mathbb{P}\left[\left(c\,\frac{\rho^2}{2}\sqrt{h(t_i-t_{i-1})}-I_h\right)^d<\frac{1}{y}\right] \\
   &= \mathbb{P}\left[I_h>y^{-1/d}\right] \leq y^{q/d}\E[|I_h|^q].
\end{align*}
 Using \eqref{akirvtx} and standard arguments, we find
 $$\E[|I_h|^q]\leq K(\E[|R_1|^q]+\E[|R_2|^q]), $$
 where
  \begin{equation*}
        R_1=\sum_{l,k,j=1}^{d}\int_{t_i-h(t_i-t_{i-1})}^{t_i}dr\int_{0}^{1}dv \,\Lambda_1^2\;\;\text{and}\;\; R_2=\sum_{l,k=1}^{d}\int_{t_i-h(t_i-t_{i-1})}^{t_i}dr\int_{0}^{1}dv \,\Lambda_2^2\,,\\
  \end{equation*}
  with
  $$\Lambda_1=\int_{r}^{t_i}\int_{0}^{1}G_{t_i-\tau}(x,z)D^{(l)}_{r,v}(\sigma_{k,j}(u(\tau,z)))W^j(d\tau,dz),$$
  and
  $$\Lambda_2=\int_{r}^{t_i}\int_{0}^{1}G_{t_i-\tau}(x,z)D^{(l)}_{r,v}(b_k(u(\tau,z)))dz\,d\tau.$$
   We bound the $q$-th moment of $R_1$ and $R_2$ separately. Concerning $R_1$, we utilize Burkholder’s inequality for Hilbert space valued martingales \cite[Eq.(4.18)]{BallyPardoux} to get
   \begin{equation}\label{R1 momenant inequality}
   \begin{split}
      &\E[|R_1|^q]\leq  \\
      & K\sum_{l,k,j=1}^{d}\E\left[\left|\int_{t_i-h(t_i-t_{i-1})}^{t_i}d\tau\int_{0}^{1}dz\,G^2_{t_i-\tau}(x,z)\int_{t_i-h(t_i-t_{i-1})}^{\tau}dr\int_{0}^{1}dv\,\Theta^2\right|^q\right],
   \end{split}
   \end{equation}
  where
  \begin{align*}
    \Theta := |D^{(l)}_{r,v}\sigma_{k,j}(u(\tau,z))| &=\left|\sum_{m=1}^{d}\frac{\partial\sigma_{k,j}}{\partial x_m}(u(\tau,z))D^{(l)}_{r,v}u_m(\tau,z)\right|\\
    &\leq K  \sum_{m=1}^{d} |D^{(l)}_{r,v}u_m(\tau,z)|,
  \end{align*}
  thanks to the hypothesis \normalfont{\textbf{A1}}. Therefore,
  \begin{equation}\label{R1 momenant inequality Phi}
   \begin{split}
      &\E[|R_1|^q]\leq  \\
      & K\sum_{l=1}^{d}\E\left[\left|\int_{t_i-h(t_i-t_{i-1})}^{t_i}d\tau\int_{0}^{1}dz\,G^2_{t_i-\tau}(x,z)\int_{t_i-h(t_i-t_{i-1})}^{\tau}dr\int_{0}^{1}dv\,\Psi^2\right|^q\right],
   \end{split}
   \end{equation}
   where $\Psi=\sum_{m=1}^{d} |D^{(l)}_{r,v}u_m(\tau,z)|$. Now we use H\"{o}lder’s inequality w.r.t. the measure $G^2_{t_i-\tau}(x,z)d\tau dz$ to obtain that
   \begin{equation*}
     \begin{split}
        \E[|R_1|^q]&\leq  K \left|\int_{t_i-h(t_i-t_{i-1})}^{t_i}d\tau\int_{0}^{1}dz\,G^2_{t_i-\tau}(x,z)\right|^{q-1} \\
          & \cdot \int_{t_i-h(t_i-t_{i-1})}^{t_i}d\tau\int_{0}^{1}dz\,G^2_{t_i-\tau}(x,z)\sum_{l=1}^{d}\E\left[\left|\int_{t_i-h(t_i-t_{i-1})}^{t_i}dr\int_{0}^{1}dv\,\Psi^2\right|^q\right].
     \end{split}
   \end{equation*}
   According to Lemma \ref{Morien} and the properties of $G$, we have
   \begin{align*}
   \E[|R_1|^q] &\leq K (h(t_i-t_{i-1}))^{\frac{q-1}{2}}(h(t_i-t_{i-1}))^{q/2}\int_{t_i-h(t_i-t_{i-1})}^{t_i}d\tau\int_{0}^{1}dz\,G^2_{t_i-\tau}(x,z)\\
   &\leq K (h(t_i-t_{i-1}))^q.
   \end{align*}

   As regards $R_2$, we derive a similar bound. By the Cauchy–Schwarz inequality,
   $$\E[|R_2|^q]\leq K(h(t_i-t_{i-1}))^q \sum_{l,k=1}^{d}\E\left[\left|\int_{t_i-h(t_i-t_{i-1})}^{t_i}dr\int_{0}^{1}dv\int_{r}^{t_i}d\tau\int_{0}^{1}dz\,\Phi^2\right|^q\right],$$
   where $\Phi=G_{t_i-\tau}(x,z)|D^{(l)}_{r,v}(b_k(u(\tau,z)))|$.  From now on, the $q$-th moment of $R_2$ is estimated as that of $R_1$ (see \eqref{R1 momenant inequality}), and this yields
$$\E[|R_2|^q]\leq K(h(t_i-t_{i-1}))^{2q}\,.$$

Hence, we have shown that
\begin{equation*}
  \E[|I_h|^q] \leq K (h(t_i-t_{i-1}))^{q}=K\,\frac{4^{2q}}{c^{2q}\,\rho^{4q}}\,y^{-2q/d}.
\end{equation*}
Consequently, taking $q>{\tilde{p}}d$,
\begin{align*}
   &\E\left[\left|(\det \Gamma^{i,i})^{-1}\right|^{{\tilde{p}}}\right] \\
   &=\int_{0}^{\infty}{\tilde{p}}\,y^{{\tilde{p}}-1}\mathbb{P}\left[(\det \Gamma^{i,i})^{-1}>y\right]dy\\
   &\leq a^{{\tilde{p}}}+{\tilde{p}}\int_{a}^{\infty}y^{{\tilde{p}}-1}\mathbb{P}\left[\det \Gamma^{i,i}<\frac{1}{y}\right]dy\\
   &\leq \frac{4^{{\tilde{p}}d}}{c^{{\tilde{p}}d}\,\rho^{2{\tilde{p}}d}(t_i-t_{i-1})^{{\tilde{p}}d/2}}+{\tilde{p}}\int_{a}^{\infty}y^{{\tilde{p}}-1+(q/d)}\E[|I_h|^q]dy \\
   &\leq \frac{4^{{\tilde{p}}d}}{c^{{\tilde{p}}d}\,\rho^{2{\tilde{p}}d}(t_i-t_{i-1})^{{\tilde{p}}d/2}}+{\tilde{p}}K\frac{4^{2q}}{c^{2q}\,\rho^{4q}}\int_{a}^{\infty}y^{{\tilde{p}}-1+(q/d)-2(q/d)}dy\\
   &\leq K' (t_i-t_{i-1})^{-{\tilde{p}}d/2},
\end{align*}
where $K'$ is a finite positive constant. Thus, we have proved that
\begin{equation}\label{det Gamma inverse s moment}
  \E\left[\left|(\det \Gamma^{i,i})^{-1}\right|^{{\tilde{p}}}\right]\leq K' (t_i-t_{i-1})^{-{\tilde{p}}d/2}.
\end{equation}

Now, turning to the second term in \eqref{def norm Malliavn Gamma ii}, we claim that for any $l=1,\cdots,k$
\begin{equation}\label{def norm Malliavn Gamma ii Dl}
       \E\left[\left\|D^l(\det \Gamma^{i,i})^{-1}\right\|^{{\tilde{p}}}_{\HH^{\otimes l}}\right]\leq K (t_i-t_{i-1})^{-{\tilde{p}}d/2},
  \end{equation}
for some positive finite constant $K$. Indeed, by iterating the equality (see \cite[Lemma 2.1.6]{Nualart})
$$D\left((\det \Gamma^{i,i})^{-1}\right)=-(\det \Gamma^{i,i})^{-2}D(\det \Gamma^{i,i}),$$
we get
\begin{equation*}
  \begin{split}
      &D^l\left((\det \Gamma^{i,i})^{-1}\right) \\
       &=\sum_{r=1}^{l}\sum_{l_1+\cdots+l_r=l \atop l_k\geq 1,\; k=1,\cdots,r}c_{r,l_1,\cdots,l_r}(\det \Gamma^{i,i})^{-(r+1)}D^{l_1}(\det \Gamma^{i,i})\otimes\cdots\otimes D^{l_r}(\det \Gamma^{i,i}).
  \end{split}
\end{equation*}
Using H\"{o}lder’s inequality, we obtain
\begin{equation}\label{Dl det Gamm moins 1 global}
  \begin{split}
     &\E\left[\left\|D^l(\det \Gamma^{i,i})^{-1}\right\|^{{\tilde{p}}}_{\HH^{\otimes l}}\right] \leq K  \sum_{r=1}^{l}\sum_{l_1+\cdots+l_r=l \atop l_k\geq 1,\; k=1,\cdots,r} \E\left[\left|(\det \Gamma^{i,i})^{-1}\right|^{{\tilde{p}}(r+1)^2}\right]^{1/(r+1)} \\
       & \times \E\left[\left\|D^{l_1}(\det \Gamma^{i,i})\right\|^{{\tilde{p}}(r+1)}_{\HH^{\otimes l_1}}\right]^{1/(r+1)}\times\cdots\times\E\left[\left\|D^{l_r}(\det \Gamma^{i,i})\right\|^{{\tilde{p}}(r+1)}_{\HH^{\otimes l_r}}\right]^{1/(r+1)}.
  \end{split}
\end{equation}

According to \eqref{det Gamma inverse s moment}, we have
\begin{equation}\label{det Gamm moins 1}
  \E\left[\left|(\det \Gamma^{i,i})^{-1}\right|^{{\tilde{p}}(r+1)^2}\right]^{1/(r+1)}\leq K (t_i-t_{i-1})^{-{\tilde{p}}(r+1)d/2},
\end{equation}
for some constant $K>0$.

For the other factors, we write
$$\det \Gamma^{i,i}= \sum_{\eta \in \Pi}\varepsilon(\eta) \prod_{k=1}^{d}\Gamma^{i,i}_{k,\eta(k)},$$
where $\Pi=\{\eta\,; \eta \text{ permutation of }\{1,\cdots,d\}\}$. Therefore
\begin{align*}
 &\E\left[\left\|D^{l}(\det \Gamma^{i,i})\right\|^{{\tilde{p}}}_{\HH^{\otimes l}}\right]  \\
 &\leq K  \sum_{\eta\in \Pi} \sum_{l_1+\cdots+l_d=l \atop l_r\geq 0,\; r=1,\cdots,d}\E\left[\left\|D^{l_1} (\Gamma^{i,i}_{1,\eta(1)})\right\|^{d{\tilde{p}}}_{\HH^{\otimes l_1}}\right]^{1/d}\times\cdots\\ &\qquad\qquad\qquad\qquad\qquad\qquad\times\E\left[\left\|D^{l_d}(\Gamma^{i,i}_{d,\eta(d)})\right\|^{d{\tilde{p}}}_{\HH^{\otimes l_d}}\right]^{1/d}.
\end{align*}
Combining \eqref{Gamma i j k l}, Lemma \ref{inner product leq norms}, and Proposition \ref{Malliavin norm increment SHE}, we get
\begin{equation*}
  \E\left[\left\|D^{l}(\det \Gamma^{i,i})\right\|^{{\tilde{p}}}_{\HH^{\otimes l}}\right]\leq K (t_i-t_{i-1})^{{\tilde{p}}d/2}.
\end{equation*}
Then
\begin{equation}\label{Dl inferieur}
  \E\left[\left\|D^{l}(\det \Gamma^{i,i})\right\|^{{\tilde{p}}(r+1)}_{\HH^{\otimes l}}\right]^{1/(r+1)}\leq K (t_i-t_{i-1})^{{\tilde{p}}d/2},
\end{equation}
where $K$ is a positive constant.

By \eqref{Dl det Gamm moins 1 global}, \eqref{det Gamm moins 1}, and \eqref{Dl inferieur}, we obtain \eqref{def norm Malliavn Gamma ii Dl}. Finally, substituting \eqref{det Gamma inverse s moment} and \eqref{def norm Malliavn Gamma ii Dl} into \eqref{def norm Malliavn Gamma ii}, we conclude the proof of Theorem \ref{thm estimation det moins 1}.
\qed\end{proof}
\begin{lem}\label{lem H t1 tn inequality for Z}
  Assume \normalfont{\textbf{A1}} and \normalfont{\textbf{A2}}. Let $\pi_n=(t_1,,\cdots,t_n)$ with $0=t_0<t_1<\cdots<t_n\leq T$, $x\in (0,1)$, $Z$ given by \eqref{Z}, and $\beta=(\beta_1,\cdots,\beta_m)$ with $\beta_{\theta}=(i_{\theta},k_{\theta})\in \{1,\cdots,n\}\times\{1,\cdots,d\}$, for $\theta=1,\cdots,m$, then there exists a constant $C>0$ such that
  \begin{equation}\label{H t1 tn inequality for Z}
       \left\|H^{\beta}_{\pi_n}(Z\,,\,1)\right\|_{0,2} \leq C\, \prod_{\theta=1}^{m} (t_{i_{\theta}}-t_{i_{\theta}-1})^{-1/4}.
  \end{equation}
\end{lem}
\begin{proof}
  By Lemma \ref{H estimation} we have
  \begin{equation}\label{H t1 tn inequality1}
    \begin{split}
       \left\|H^{\beta}_{\pi_n}(Z\,,\,1)\right\|_{0,2} & \leq C\,\left\|\left(\det{\Gamma_{Z\,,\,\pi_n}}\right)^{-1}\right\|^m_{m,2^{m+2}}  \\
         &\times \prod_{\theta=1}^{m}\left\|D(u_{k_{\theta}}(t_{i_{\theta}},x)-u_{k_{\theta}}(t_{i_{\theta}-1},x))\right\|_{m,2^{2(m+nd)},\HH}\\
         &\times\prod_{(i_0,k_0)\in O_{(i_{\theta},k_{\theta})}}\left\|D(u_{k_{0}}(t_{i_{0}},x)-u_{k_{0}}(t_{i_{0}-1},x))\right\|^2_{m,2^{2(m+nd)},\HH},
    \end{split}
  \end{equation}
  where $O_{(i_{\theta},k_{\theta})}$ is given in Lemma \ref{H estimation}. Hence, according to Proposition \ref{Malliavin norm increment SHE} and Theorem \ref{thm estimation det moins 1}, we get
  \begin{equation}\label{H t1 tn inequality12}
    \begin{split}
       \left\|H^{\beta}_{\pi_n}(Z\,,\,1)\right\|_{0,2} & \leq \tilde{C} \, \prod_{j=1}^{n}(t_j-t_{j-1})^{-md/2} \prod_{\theta=1}^{m}(t_{i_{\theta}}-t_{i_{\theta}-1})^{1/4}\\
         &\qquad\qquad\times\prod_{(i_0,k_0)\in O_{(i_{\theta},k_{\theta})}}(t_{i_{0}}-t_{i_{0}-1})^{1/2}\\
         &= \tilde{C} \, \prod_{j=1}^{n}(t_j-t_{j-1})^{-md/2} \prod_{\theta=1}^{m}(t_{i_{\theta}}-t_{i_{\theta}-1})^{-1/4}(t_{i_{\theta}}-t_{i_{\theta}-1})^{1/2}\\
         &\qquad\qquad\times\prod_{(i_0,k_0)\in O_{(i_{\theta},k_{\theta})}}(t_{i_{0}}-t_{i_{0}-1})^{1/2}\\
         &= \tilde{C} \, \prod_{j=1}^{n}(t_j-t_{j-1})^{-md/2} \prod_{\theta=1}^{m}(t_{i_{\theta}}-t_{i_{\theta}-1})^{-1/4}\\
         &\qquad\qquad\times\prod_{\theta=1}^{m}(t_{i_{\theta}}-t_{i_{\theta}-1})^{1/2}\prod_{(i_0,k_0)\in O_{(i_{\theta},k_{\theta})}}(t_{i_{0}}-t_{i_{0}-1})^{1/2}.
    \end{split}
  \end{equation}
We remark that $(t_{i_{\theta}}-t_{i_{\theta}-1})^{1/2}\prod_{(i_0,k_0)\in O_{(i_{\theta},k_{\theta})}}(t_{i_{0}}-t_{i_{0}-1})^{1/2}$ is equal to
$\prod_{j=1}^{n}\prod_{k=1}^{d}(t_j-t_{j-1})^{1/2}=\prod_{j=1}^{n}(t_j-t_{j-1})^{d/2}$. Therefore,
\begin{equation}\label{eq H}
  \prod_{\theta=1}^{m}(t_{i_{\theta}}-t_{i_{\theta}-1})^{1/2}\prod_{(i_0,k_0)\in O_{(i_{\theta},k_{\theta})}}(t_{i_{0}}-t_{i_{0}-1})^{1/2}= \prod_{j=1}^{n}(t_j-t_{j-1})^{md/2}.
\end{equation}
Hence, by combining \eqref{H t1 tn inequality12} and   \eqref{eq H}, we conclude the proof of Lemma \ref{lem H t1 tn inequality for Z}.
\qed\end{proof}

\begin{proof}[Proof of Theorem  \ref{estimation density}\textbf{(b)}]
  By Theorem \ref{existence of density general}, Lemma \ref{lem H t1 tn inequality for Z}, and \cite[Ineq. (6.2)]{DalangKhoshnevisanNualartmultiplicative}, we conclude the proof.
\qed\end{proof}
\section{Proof of Theorem \ref{local time SHE}}\label{4}
In this section, we will investigate the existence of the local time and its joint continuous version for the process $\{u(t,x)\,,\;t\in [0,T]\}$.
\subsection{Existence of local time when $d\leq 3$}
Let $\alpha\geq 0$, we define the Sobolev space $H^{\alpha}(\R^d)$ as:
$$H^{\alpha}(\R^d)=\left\{g\in L^2(\R^d)\,;\;(1+\|\xi\|^2)^{\frac{\alpha}{2}}\hat{g}\in L^2(\R^d)\right\},$$
where $\|\cdot\|$ is the Euclidean norm on  $\R^d$ and $\hat{g}$ is the Fourier transform of $g$.

Now we give our result concerning the existence of local time of the solution to Eq. \eqref{SHE}.
\begin{thm}\label{thm existence of local time}
  Let $u(t,x)$ be given by \eqref{solution SHE}. Assume that $d\leq 3$, then for each $x\in (0,1)$, the process $\{u(t,x)\,,\;t\in[0,T]\}$ has a local time $L(\xi,t)$. Moreover, for every fixed $t$, $L(\bullet,t)\in H^{\alpha}(\R^d)$ for $\alpha<\frac{4-d}{2}$.
\end{thm}
\begin{proof}
  Let $t\in [0,T]$, and define $f$ by
  $$f(\xi)=\int_{0}^{t}e^{i\left<\xi,u(s,x)\right>}ds.$$
  Note that $f$ coincides with the Fourier transform of the local time ,$L(\bullet,t)$, whenever $L(\bullet,t)$ exists.
  Since $f$ is a continuous function, then we have just to look for $\alpha$ such that
  $$\int_{\R^d}\|\xi\|^{2\alpha}|f(\xi)|^2d\xi<\infty.$$
  We have by Fubini's theorem
  \begin{align*}
    \E\left[\int_{\R^d}\|\xi\|^{2\alpha}|f(\xi)|^2d\xi\right] &= \int_{\R^d}\|\xi\|^{2\alpha}\int_{[0,t]^2}\E\left[e^{i\left<\xi,u(s,x)-u(r,x)\right>}\right]dr\,ds\, d\xi \\
    &\leq \int_{[0,t]^2}\int_{\R^d}\|\xi\|^{2\alpha}\left|\E\left[e^{i\left<\xi,u(s,x)-u(r,x)\right>}\right]\right| d\xi\,dr\,ds \\
    &= 2\int_{\{0\leq r<s\leq t\}} \int_{\R^d}\|\xi\|^{2\alpha}\left|\E\left[e^{i\left<\xi,u(s,x)-u(r,x)\right>}\right]\right| d\xi\,dr\,ds.
  \end{align*}
Put $I_1=[-1/(s-r)^{1/4},1/(s-r)^{1/4}]$ and $I_2=\R\setminus I_1$, Therefore
$$\R^d=\bigcup_{i_1,\cdots,i_d\in \{1,2\}}I_{i_1}\times\cdots\times I_{i_d}.$$
Hence
\begin{equation}\label{E 1 plus xi s }
\begin{split}
  &\E\left[\int_{\R^d}\|\xi\|^{2\alpha}|f(\xi)|^2d\xi\right]\\
  &\leq 2\sum_{i_1,\cdots,i_d\in \{1,2\}}\int_{\{0\leq r<s\leq t\}} \int_{I_{i_1}\times\cdots\times I_{i_d}} \|\xi\|^{2\alpha}\left|\E\left[e^{i\left<\xi,u(s,x)-u(r,x)\right>}\right]\right| d\xi\,dr\,ds.
  \end{split}
\end{equation}

Let $\varphi_{\xi}:\R^d\ni z\mapsto e^{i\left<\xi,z\right>}$, $k=(k_1,\ldots,k_d)$, and $\partial_{z}^{k}=\prod_{l=1}^{d}(\frac{\partial}{\partial z_l})^{k_l}$. By simple calculation, we have
\begin{equation}\label{carac fc derivatives inside}
  \E\left[(\partial_{z}^{k}\varphi_{\xi})(u(s,x)-u(r,x))\right]=(i\xi_1)^{k_1}\cdots(i\xi_d)^{k_d}\E\left[e^{i\left<\xi,u(s,x)-u(r,x)\right>}\right]
\end{equation}
On the other hand, by integration by parts  (i.e., Proposition \ref{prop integration by parts}), we get
 \begin{align}\label{carac fc int parts}
  \E\left[(\partial_{z}^{k}\varphi_{\xi})(u(s,x)-u(r,x))\right] &= \E\left[(\partial_{y_2}^{k}g)(u(r,x),u(s,x)-u(r,x))\right]\nonumber\\
  &= \E\left[e^{i\left<\xi,u(s,x)-u(r,x)\right>}H_{\pi_2}^{\beta}(F,1)\right],
\end{align}
where $\pi_2=(r,s)$, $F=(u(r,x),u(s,x)-u(r,x))$, $g:\R^d\times\R^d\ni(y_1,y_2)\mapsto\varphi_{\xi}(y_2)$, $y_i=(y_{i,1},\ldots,y_{i,d})$, for $i=1,2$, $\partial_{y_2}^{k}=\prod_{l=1}^{d}(\frac{\partial}{\partial y_{2,l}})^{k_l}$, $m=\sum_{l=1}^{d}k_l$, and
$$\beta=\left(\underbrace{(2,1),\cdots,(2,1)}_{k_1 \text{times}},\cdots,\underbrace{(2,d),\cdots,(2,d)}_{k_d \text{times}}\right).$$
Combining \eqref{carac fc derivatives inside} and \eqref{carac fc int parts}, we obtain
\begin{align}\label{cara fc est}
  \left|\E\left[e^{i\left<\xi,u(s,x)-u(r,x)\right>}\right]\right| &= |\xi_1|^{-k_1}\cdots|\xi_d|^{-k_d} \left|\E\left[e^{i\left<\xi,u(s,x)-u(r,x)\right>}H_{\pi_2}^{\beta}(F,1)\right]\right| \nonumber\\
   &\leq |\xi_1|^{-k_1}\cdots|\xi_d|^{-k_d} \E\left[\left|H_{\pi_2}^{\beta}(F,1)\right|\right].
\end{align}
According to \eqref{H t1 tn inequality for Z}, we write
\begin{equation}\label{estimatin H s r}
  \|H_{\pi_2}^{\beta}(F,1)\|_{0,2} \leq C\prod_{\theta=1}^{m}(s-r)^{-1/4}= C (s-r)^{-\sum_{l=1}^{d}k_l/4}.
\end{equation}
Therefore, according to \eqref{cara fc est} and \eqref{estimatin H s r}, we get for all positive integers $k_1,\cdots,k_d$, there exists a positive constant $C=C(k_1,\cdots,k_d)$, such that
\begin{equation}\label{alpha LND for 1}
  \left|\E\left[e^{i\left<\xi,u(s,x)-u(r,x)\right>}\right]\right|\leq \frac{C}{|\xi_1|^{k_1}\cdots|\xi_d|^{k_d} (s-r)^{\sum_{l=1}^{d}k_l/4}},
\end{equation}
here $\xi=(\xi_1,\cdots,\xi_d)$. Put, for $l=1,\cdots,d$,
$$k_l(i_l)=\left\{
  \begin{array}{ll}
    0, & \text{if}\quad i_l=1; \\
    2([\alpha]+2), & \text{if}\quad i_l=2,
  \end{array}
\right.$$
where $[\alpha]$ is the integer part of $\alpha$. According to \eqref{E 1 plus xi s } and \eqref{alpha LND for 1}, we get
\begin{align}\label{trop f}
  &\E\left[\int_{\R^d}\|\xi\|^{2\alpha}|f(\xi)|^2d\xi\right]\nonumber\\
  &\leq K_1\sum_{i_1,\cdots,i_d\in \{1,2\}}\int_{\{0\leq r<s\leq t\}} \int_{I_{i_1}\times\cdots\times I_{i_d}} \nonumber\\ &\qquad\qquad\qquad\qquad\qquad\qquad\frac{\|\xi\|^{2\alpha}}{|\xi_1|^{k_1(i_1)}\cdots|\xi_d|^{k_d(i_d)} (s-r)^{\sum_{l=1}^{d}k_l(i_l)/4}} d\xi\,dr\,ds\nonumber\\
  &\leq K_2\sum_{p=1}^{d}\sum_{i_1,\cdots,i_d\in \{1,2\}}\int_{\{0\leq r<s\leq t\}} \int_{I_{i_1}\times\cdots\times I_{i_d}} \nonumber\\ &\qquad\qquad\qquad\qquad\qquad\qquad\frac{|\xi_p|^{2\alpha}}{|\xi_1|^{k_1(i_1)}\cdots|\xi_d|^{k_d(i_d)} (s-r)^{\sum_{l=1}^{d}k_l(i_l)/4}} d\xi\,dr\,ds\nonumber\\
  &=K_2\sum_{p=1}^{d}\sum_{i_1,\cdots,i_d\in \{1,2\}}\int_{\{0\leq r<s\leq t\}} \prod_{l=1}^{p-1}\int_{I_{i_l}}\frac{1}{|\xi_l|^{k_l(i_l)}(s-r)^{k_l(i_l)/4}}d\xi_l\nonumber\\
  &\qquad\times \int_{I_{i_p}}\frac{|\xi_p|^{2\alpha}}{|\xi_p|^{k_p(i_p)}(s-r)^{k_p(i_p)/4}}d\xi_p\prod_{l=p+1}^{d}\int_{I_{i_l}}\frac{1}{|\xi_l|^{k_l(i_l)}(s-r)^{k_l(i_l)/4}}d\xi_l\,dr\,ds.
\end{align}
By simple calculation, we obtain
\begin{itemize}
  \item When $i_l=1\text{ or }2$ with $l\neq p$, we have $\int_{I_{i_l}}\frac{1}{|\xi_l|^{k_l(i_l)}(s-r)^{k_l(i_l)/4}}d\xi_l= \frac{c}{(s-r)^{1/4}}$, where $c$ is a positive constant depending on $i_l$ and $\alpha$ ;
  \item If $i_p=1\text{ or }2$, we get $\int_{I_{i_p}}\frac{|\xi_p|^{2\alpha}}{|\xi_p|^{k_p(i_p)}(s-r)^{k_p(i_p)/4}}d\xi_p=\frac{c'}{(s-r)^{(2\alpha+1)/4}}$, where $c'$ is a positive constant depending on $i_p$ and $\alpha$.
\end{itemize}
Combining the above discussion with \eqref{trop f}, we have
\begin{equation}\label{Ineq f s r}
    \E\left[\int_{\R^d}\|\xi\|^{2\alpha}|f(\xi)|^2d\xi\right] \leq K_3 \int_{\{0\leq r<s\leq t\}}\frac{1}{(s-r)^{(2\alpha+d)/4}}dr\,ds.
\end{equation}
Hence, the local time $L(\xi,t)$ exists for $d=1,2,3$ and $L(\bullet,t)\in H^{\alpha}(\R^d)$ for $\alpha<\frac{4-d}{2}$. Which finishes the proof of Theorem \ref{thm existence of local time}.
\qed\end{proof}
\subsection{The local time does not exist for $d\geq 4$}\label{42}
The below theorem is a classical result on the existence of local time for a stochastic process $X$ with values in $\R^d$.
\begin{thm}[Theorem 21.15 in \cite{GemanHorowitz}]\label{No LT}
  The local time, $L(\bullet,t)$, exists with $L(\bullet,t)\in L^2(\mathbb{P}\otimes \lambda_d)$ iff
  $$\liminf_{\varepsilon\to 0}\varepsilon^{-d}\int_{0}^{t}\int_{0}^{t}\mathbb{P}\left[\|X_s-X_r\|\leq \varepsilon\right]dr\,ds<\infty.$$
\end{thm}

Taking into account the above theorem, we are ready to give the following result.
\begin{thm}\label{LT does not exist for u}
  Let $u(t,x)$ be given by \eqref{solution SHE}. Assume that $d\geq 4$, then for each $x\in (0,1)$, the process $\{u(t,x)\,,\;t\in[0,T]\}$ does not have a local time $L(\xi,t)$ in $L^2(\mathbb{P}\otimes \lambda_d)$ for any $t\in [0,T]$.
\end{thm}
\begin{proof}
  We have by \eqref{lower bound for the density} and Fubini's theorem
  \begin{align*}
     &\int_{0}^{t}\int_{0}^{t}\mathbb{P}\left[\|u(s,x)-u(r,x)\|\leq \varepsilon\right]dr\,ds \\
     &\geq c\int_{B(0,\varepsilon)} \int_{0}^{t}\int_{0}^{t}\frac{1}{|s-r|^{d/4}}\exp\left(-\frac{\|y\|^2}{c|s-r|^{1/2}}\right)dr\,ds\,dy. \\
  \end{align*}
  We now fix $s$ and use the change of variables $\tau=s-r$ to see that this above expression equal to
  \begin{align*}
     &c\int_{B(0,\varepsilon)} \int_{0}^{t}\int_{-t+s}^{s}\frac{\exp\left(-\frac{\|y\|^2}{c|\tau|^{1/2}}\right)}{|\tau|^{d/4}}d\tau\,ds\,dy. \\
  \end{align*}
Let $0<\alpha<t$, hence this above term is greater than or equal to
\begin{align*}
   &c\int_{B(0,\varepsilon)} \int_{t-\alpha}^{t}\int_{0}^{s}\frac{\exp\left(-\frac{\|y\|^2}{c\tau^{1/2}}\right)}{\tau^{d/4}}d\tau\,ds\,dy  \\
   &\geq c\int_{B(0,\varepsilon)} \int_{t-\alpha}^{t}\int_{0}^{t-\alpha}\frac{\exp\left(-\frac{\|y\|^2}{c\tau^{1/2}}\right)}{\tau^{d/4}}d\tau\,ds\,dy  \\
   &=c_1\int_{B(0,\varepsilon)}\int_{0}^{t-\alpha}\frac{\exp\left(-\frac{\|y\|^2}{c\tau^{1/2}}\right)}{\tau^{d/4}}d\tau\,dy.
\end{align*}
By the change of variables  $\tau=c^{-2}\|y\|^4u$ we see that this is greater than or equal to
\begin{equation}\label{y d mmoins 4}
  c_2\int_{B(0,\varepsilon)}\frac{1}{\|y\|^{d-4}}\int_{0}^{\frac{c^2(t-\alpha)}{\|y\|^4}}\frac{\exp\left(-\frac{1}{u^{1/2}}\right)}{u^{d/4}}du\,dy.
\end{equation}
\begin{itemize}
  \item If $d\geq 5$. Assume that $\varepsilon\in (0,1)$ and let $0<\beta<c^2(t-\alpha)$.  Then \eqref{y d mmoins 4} is greater than or equal to
  \begin{align*}
     c_2\int_{B(0,\varepsilon)}\frac{1}{\|y\|^{d-4}}\int_{\beta}^{c^2(t-\alpha)}\frac{\exp\left(-\frac{1}{u^{1/2}}\right)}{u^{d/4}}du\,dy &=c_3 \int_{B(0,\varepsilon)}\frac{1}{\|y\|^{d-4}}dy\geq c_3 \varepsilon^4. \\
  \end{align*}
  Therefore, we conclude that
  $$\liminf_{\varepsilon\to 0}\varepsilon^{-d}\int_{0}^{t}\int_{0}^{t}\mathbb{P}\left[\|X_s-X_r\|\leq \varepsilon\right]dr\,ds=\infty.$$
  Hence, by Theorem \ref{No LT} the local time does not exist for $d\geq 5$.
  \item If $d=4$. Assume that $\varepsilon\in (0,1)$ and let $0<\beta<c^2(t-\alpha)$.  Then \eqref{y d mmoins 4} is greater than or equal to
  \begin{align*}
     c_2\int_{B(0,\varepsilon)}\int_{\beta}^{\frac{c^2(t-\alpha)}{\|y\|^4}}\frac{\exp\left(-\frac{1}{u^{1/2}}\right)}{u}du\,dy& \geq c_2 e^{-\frac{1}{\beta^{1/2}}}\int_{B(0,\varepsilon)} \log\left(\frac{c^2(t-\alpha)}{\beta\|y\|^4}\right) dy  \\
     &\geq c_2e^{-\frac{1}{\beta^{1/2}}} \varepsilon^4 \log\left(\frac{c^2(t-\alpha)}{\beta\varepsilon^4}\right).
  \end{align*}
  Therefore
  $$\liminf_{\varepsilon\to 0}\varepsilon^{-4}\int_{0}^{t}\int_{0}^{t}\mathbb{P}\left[\|X_s-X_r\|\leq \varepsilon\right]dr\,ds=\infty.$$
  Then, by Theorem \ref{No LT} the local time does not exist for $d=4$. This concludes the proof of Theorem \ref{LT does not exist for u}.
\end{itemize}
\qed\end{proof}
\subsection{Regularity of local time}
Our goal in this section is to look for a version of the local time $L(\xi,t)$ with jointly
H\"{o}lder continuity in $(\xi,t)$. Moreover, we show that the local time satisfies a H\"{o}lder condition with respect to the time variable $t$, uniformly in the space variable $\xi$.  We start by proving the $\alpha$-LND property  for the process $\{u(t,x)\,,\;t\in[0,T]\}$.
\begin{thm}\label{thm alpha LND}
 Let $u(t,x)$ be given by \eqref{solution SHE}. Hence, for each fixed $x\in (0,1)$, the process $\{u(t,x)\,,\;t\in[0,T]\}$ verifies the $\frac{1}{4}$-LND property on $[0,T]$, i.e.,  for every nonnegative integers $m\geq 2$, $k_{j,l}$, for $j=1,\cdots,m$ and $l=1,\cdots,d$, there exists a constant $c=c(m,k_{j,l})$ such that
  \begin{equation}\label{Ineq alphha LND for u}
        \left|\E\left[e^{i\sum_{j=1}^{m}\left<v_j,u(t_j,x)-u(t_{j-1},x)\right>}\right]\right|
         \leq\frac{c}{\prod_{j=1}^{m}\prod_{l=1}^{d}|v_{j,l}|^{k_{j,l}}(t_j-t_{j-1})^{k_{j,l}/4}},
  \end{equation}
  for all $v_j=(v_{j,l}\,;\; 1\leq l\leq d)\in (\R\setminus\{0\})^d$, for $j=1,\cdots, m$, and for every ordered points $0=t_0<t_1<\cdots<t_m\leq T$.
\end{thm}
\begin{proof}
  The proof is a simple consequence of the integration by parts  \eqref{integration by parts} and \eqref{H t1 tn inequality for Z}.
\qed\end{proof}

In order to use Kolmogorov's theorem to conclude various continuities of the local time $L(\xi,t)$ in $t$ and $\xi$, we seek to estimate the moments of the increments of $L(\xi,t)$.
\begin{lem}\label{lem Kolmogorov}
 Let $u(t,x)$ be given by \eqref{solution SHE}. Assume $d\leq 3$. Let $\tilde{L}(\xi,t)$ be given as in \eqref{L tilde}, therefore, for every $\xi,y\in \R^d$, $t,t+h\in [0,T]$, and even integer $m\geq 2$,
 \begin{equation}\label{t t plus h}
   \E\left[\tilde{L}(\xi,t+h)-\tilde{L}(\xi,t)\right]^{m}\leq C_m |h|^{m(1-\frac{d}{4})};
 \end{equation}
 \begin{equation}\label{t t plus h xi xi plus y}
 \begin{split}
     & \E\left[\tilde{L}(\xi+y,t+h)-\tilde{L}(\xi,t+h)-\tilde{L}(\xi+y,t)+\tilde{L}(\xi,t)\right]^{m} \\
      & \leq C_{m,\theta} \|y\|^{m\theta}|h|^{m(1-\frac{d}{4}-\frac{\theta}{4})},
 \end{split}
 \end{equation}
 where $0<\theta<(\frac{4-d}{2})\wedge1$.
\end{lem}
\begin{proof}
  We prove just the second inequality; the first one follows the same lines. We consider
only $h>0$ such that $t+h\in [0,T]$, the other case follows the same way. According to \eqref{E Lxkth Lxth Lxkt Lxt}, we get

\begin{equation*}
    \begin{split}
    &\E[\tilde{L}(\xi+y,t+h)-\tilde{L}(\xi,t+h)-\tilde{L}(\xi+y,t)+\tilde{L}(\xi,t)]^m \\
         &=\frac{1}{(2\pi)^{md}}\int_{(\R^d)^m}\int_{[t,t+h]^m}\prod_{j=1}^{m}\left(e^{-i\left<v_j-v_{j+1},\xi+y\right>}-e^{-i\left<v_j-v_{j+1},\xi\right>}\right)\\
         &\qquad\qquad\qquad\qquad\qquad\qquad\times\E\left[e^{i\sum_{j=1}^{m}\left<v_j,u(t_j,x)-u(t_{j-1},x)\right>}\right]\prod_{j=1}^{m}dt_j\prod_{j=1}^{m}dv_j.
    \end{split}
\end{equation*}
By the elementary inequality $|1-e^{i\rho}|\leq 2^{1-\theta}|\rho|^{\theta}$ for all $0<\theta<1$ and $\rho\in \R$, we have
\begin{equation}\label{t t plus h xi xi plus y J}
  \begin{split}
       & \E[\tilde{L}(\xi+y,t+h)-\tilde{L}(\xi,t+h)-\tilde{L}(\xi+y,t)+\tilde{L}(\xi,t)]^m \\
       & \leq 2^{-md-\theta+1}\pi^{-md}\|y\|^{m\theta}\mathcal{J}(m,\theta),
  \end{split}
\end{equation}
where
\begin{equation*}
  \begin{split}
      &\mathcal{J}(m,\theta)  \\
       & =\int_{[t,t+h]^m}\int_{(\R^d)^m}\prod_{j=1}^{m}\|v_j-v_{j+1}\|^{\theta}\left|\E\left[e^{i\sum_{j=1}^{m}\left<v_j,u(t_j,x)-u(t_{j-1},x)\right>}\right]\right|\prod_{j=1}^{m}dv_j\prod_{j=1}^{m}dt_j.
  \end{split}
\end{equation*}
We replace the integration over the domain $[t,t+h]^m$ by the integration over the
subset  $\Lambda=\{t\leq t_1<\cdots<t_m\leq t+h\}$, hence we obtain
\begin{equation*}
  \begin{split}
      &\mathcal{J}(m,\theta)  \\
       & =m!\int_{\Lambda}\int_{(\R^d)^m}\prod_{j=1}^{m}\|v_j-v_{j+1}\|^{\theta}\left|\E\left[e^{i\sum_{j=1}^{m}\left<v_j,u(t_j,x)-u(t_{j-1},x)\right>}\right]\right|\prod_{j=1}^{m}dv_j\prod_{j=1}^{m}dt_j,
  \end{split}
\end{equation*}
where $t_0=0$ and $v_{m+1}=0$. By the fact that $\|a-b\|^{\theta} \leq \|a\|^{\theta}+\|b\|^{\theta}$ for all $0<\theta<1$ and $a,b\in \R^d$, it follows that
\begin{equation}\label{}
  \prod_{j=1}^{m}\|v_j-v_{j+1}\|^{\theta} \leq  \prod_{j=1}^{m} \left(\|v_j\|^{\theta}+\|v_{j+1}\|^{\theta}\right).
\end{equation}
 Note that the right side of this last inequality is at most equal to a finite sum of terms each of the
form $\prod_{j=1}^{m}\|v_j\|^{\epsilon_j \theta}$, where $\epsilon_j=0,1,$ or $2$ and $\sum_{j=1}^{m}\epsilon_j=m$. Therefore
\begin{equation}\label{J without alpha LND}
  \begin{split}
      &\mathcal{J}(m,\theta)\leq m! \sum_{(\epsilon_1,\cdots,\epsilon_m)\in \{0,1,2\}^m}\int_{\Lambda}\int_{(\R^d)^m}\prod_{j=1}^{m}\|v_j\|^{\epsilon_j\theta}\\
       &\qquad\qquad\qquad\qquad\qquad\qquad\quad\times\left|\E\left[e^{i\sum_{j=1}^{m}\left<v_j,u(t_j,x)-u(t_{j-1},x)\right>}\right]\right|\prod_{j=1}^{m}dv_j\prod_{j=1}^{m}dt_j.
  \end{split}
\end{equation}
On the other hand, by the $\frac{1}{4}$-LND property, i.e., Theorem  \ref{thm alpha LND} we get for every nonnegative integers $m\geq 2$, $k_{j,l}$, for $j=1,\cdots,m$ and $l=1,\cdots,d$, there exists a constant $c=c(m,k_{j,l})$ such that
  \begin{equation}\label{Ineq alphha LND for u 2}
        \left|\E\left[e^{i\sum_{j=1}^{m}\left<v_j,u(t_j,x)-u(t_{j-1},x)\right>}\right]\right|
         \leq\frac{c}{\prod_{j=1}^{m}\prod_{l=1}^{d}|v_{j,l}|^{k_{j,l}}(t_j-t_{j-1})^{k_{j,l}/4}},
  \end{equation}
where $v_j=(v_{j,1},\cdots,v_{j,d})$. Put $I^j_1=[-1/(t_j-t_{j-1})^{1/4},1/(t_j-t_{j-1})^{1/4}]$ and $I^j_2=\R\setminus I^j_1$, Therefore
\begin{equation}\label{R d m}
  (\R^d)^m=\bigcup_{i_{j,l}\in \{1,2\}\atop j=1,\cdots,m; l=1,\cdots,d}\prod_{j=1}^{m} I^j_{i_{j,1}}\times\cdots\times I^j_{i_{j,d}}.
\end{equation}
Set, for $j=1,\cdots,m$ and $l=1,\cdots,d$,
$$k_{j,l}(i_{j,l})=\left\{
  \begin{array}{ll}
    0, & \text{if}\quad i_{j,l}=1; \\
    4, & \text{if}\quad i_{j,l}=2,
  \end{array}
\right.$$
Hence, by \eqref{J without alpha LND}, \eqref{Ineq alphha LND for u 2}, and \eqref{R d m}, we obtain
\begin{equation}\label{}
  \begin{split}
      \mathcal{J}(m,\theta)&\leq m!c\sum_{i_{j,l}\in \{1,2\}\atop j=1,\cdots,m; l=1,\cdots,d} \sum_{(\epsilon_1,\cdots,\epsilon_m)\in \{0,1,2\}^m}\int_{\Lambda}\int_{\prod_{j=1}^{m} I^j_{i_{j,1}}\times\cdots\times I^j_{i_{j,d}}}  \\
       & \times \frac{\prod_{j=1}^{m}\|v_j\|^{\epsilon_j\theta}}{\prod_{j=1}^{m}\prod_{l=1}^{d}|v_{j,l}|^{k_{j,l}(i_{j,l})}(t_j-t_{j-1})^{k_{j,l}(i_{j,l})/4}}\prod_{j=1}^{m}dv_j\prod_{j=1}^{m}dt_j.
  \end{split}
\end{equation}
We remark that
\begin{equation*}
   \prod_{j=1}^{m}\|v_j\|^{\epsilon_j\theta} \leq \prod_{j=1}^{m} \left(|v_{j,1}|^{\epsilon_j\theta}+\cdots+|v_{j,d}|^{\epsilon_j\theta}\right) =\sum_{l_1,\cdots,l_d\in \{1,\cdots,d\}}\prod_{j=1}^{m} |v_{j,l_j}|^{\epsilon_j\theta}.
\end{equation*}
Therefore
\begin{equation*}
  \begin{split}
      \mathcal{J}(m,\theta)&\leq m!c\sum_{l_1,\cdots,l_d\in \{1,\cdots,d\}}\sum_{i_{j,l}\in \{1,2\}\atop j=1,\cdots,m; l=1,\cdots,d} \sum_{(\epsilon_1,\cdots,\epsilon_m)\in \{0,1,2\}^m}\int_{\Lambda}\int_{\prod_{j=1}^{m} I^j_{i_{j,1}}\times\cdots\times I^j_{i_{j,d}}} \\
       &\times \frac{\prod_{j=1}^{m} |v_{j,l_j}|^{\epsilon_j\theta}}{\prod_{j=1}^{m}\prod_{l=1}^{d}|v_{j,l}|^{k_{j,l}(i_{j,l})}(t_j-t_{j-1})^{k_{j,l}(i_{j,l})/4}}\prod_{j=1}^{m}dv_j\prod_{j=1}^{m}dt_j.
  \end{split}
\end{equation*}
According to Fubini's theorem, the right side of the above expression is equal to
\begin{equation*}
  \begin{split}
      & \quad\qquad m!c\sum_{l_1,\cdots,l_d\in \{1,\cdots,d\}}\sum_{i_{j,l}\in \{1,2\}\atop j=1,\cdots,m; l=1,\cdots,d} \sum_{(\epsilon_1,\cdots,\epsilon_m)\in \{0,1,2\}^m}\int_{\Lambda}\prod_{j=1}^{m}\int_{ I^j_{i_{j,1}}\times\cdots\times I^j_{i_{j,d}}} \\
       &\qquad\qquad\qquad\qquad\times \frac{ |v_{j,l_j}|^{\epsilon_j\theta}}{\prod_{l=1}^{d}|v_{j,l}|^{k_{j,l}(i_{j,l})}(t_j-t_{j-1})^{k_{j,l}(i_{j,l})/4}} dv_j\prod_{j=1}^{m}dt_j.
  \end{split}
\end{equation*}
\begin{equation}\label{J m theta prod}
  \begin{split}
      &= m!c\sum_{l_1,\cdots,l_d\in \{1,\cdots,d\}}\sum_{i_{j,l}\in \{1,2\}\atop j=1,\cdots,m; l=1,\cdots,d} \sum_{(\epsilon_1,\cdots,\epsilon_m)\in \{0,1,2\}^m}\\
      &\qquad\qquad\times\int_{\Lambda}\prod_{j=1}^{m}\prod_{l=1\atop l\neq l_j}^{d}\int_{ I^j_{i_{j,l}}} \frac{ 1}{|v_{j,l}|^{k_{j,l}(i_{j,l})}(t_j-t_{j-1})^{k_{j,l}(i_{j,l})/4}} dv_{j,l}\\
      &\qquad\qquad\times\int_{ I^j_{i_{j,l_j}}} \frac{1}{|v_{j,l_j}|^{k_{j,l_j}(i_{j,l_j})-\epsilon_j\theta}(t_j-t_{j-1})^{k_{j,l_j}(i_{j,l_j})/4}} dv_{j,l_j}\prod_{j=1}^{m}dt_j.
  \end{split}
\end{equation}
\begin{itemize}
  \item If $i_{j,l}=1$ or $2$ with $l\neq l_j$, then we have
  $$\int_{ I^j_{i_{j,l}}} \frac{ 1}{|v_{j,l}|^{k_{j,l}(i_{j,l})}(t_j-t_{j-1})^{k_{j,l}(i_{j,l})/4}} dv_{j,l}=\frac{K_1}{(t_j-t_{j-1})^{1/4}},$$
  where the constant $K_1$ depends only on $i_{j,l}$.
  \item If $i_{j,l_j}=1$ or $2$, then we get
  $$\int_{ I^j_{i_{j,l_j}}} \frac{1}{|v_{j,l_j}|^{k_{j,l_j}(i_{j,l_j})-\epsilon_j\theta}(t_j-t_{j-1})^{k_{j,l_j}(i_{j,l_j})/4}} dv_{j,l_j}=\frac{K_2}{(t_j-t_{j-1})^{(1+\epsilon_j\theta)/4}},$$
  where the constant $K_2$ depends on $i_{j,l_j}$, $\theta$, and $\epsilon_j$ such that $\sup_{\theta,\epsilon_j}K_2<\infty$.
\end{itemize}
Combining the above discussion with \eqref{J m theta prod}, we obtain
\begin{equation}\label{J prod Ehm}
  \begin{split}
     \mathcal{J}(m,\theta) &\leq  m!c_1\sum_{l_1,\cdots,l_d\in \{1,\cdots,d\}}\sum_{i_{j,l}\in \{1,2\}\atop j=1,\cdots,m; l=1,\cdots,d} \sum_{(\epsilon_1,\cdots,\epsilon_m)\in \{0,1,2\}^m}\\
       &\times\int_{\Lambda}\prod_{j=1}^{m}\frac{1}{(t_j-t_{j-1})^{(d+\epsilon_j\theta)/4}}\prod_{j=1}^{m}dt_j.
  \end{split}
\end{equation}
According to  an elementary calculation (cf. Ehm \cite{Ehm}), for every $m\geq 1$, $h>0$, and $b_j<1$,
\begin{equation}\label{Ehm equality}
  \int_{t\leq s_1<\cdots<s_m\leq t+h}\prod_{j=1}^{m}\frac{1}{(s_j-s_{j-1})^{b_j}}\prod_{j=1}^{m}ds_j=h^{m-\sum_{j=1}^{m}b_j}\frac{\prod_{j=1}^{m}\Gamma(1-b_j)}{\Gamma(1+k-\sum_{j=1}^{m}b_j)},
\end{equation}
where $s_0=t$. By \eqref{J prod Ehm} and \eqref{Ehm equality}, it follows that for $0<\theta<(\frac{4-d}{2})\wedge 1$ and $b_j=\frac{d+\epsilon_j\theta}{4}$,
\begin{equation}
     \mathcal{J}(m,\theta) \leq \tilde{C}(m,\theta) h^{m(1-\frac{d+\theta}{4})}.
\end{equation}
Finally, by \eqref{t t plus h xi xi plus y J} we get
\begin{equation*}
        \E[\tilde{L}(\xi+y,t+h)-\tilde{L}(\xi,t+h)-\tilde{L}(\xi+y,t)+\tilde{L}(\xi,t)]^m \\
        \leq C(m,\theta)\|y\|^{m\theta}h^{m(1-\frac{d+\theta}{4})}.
\end{equation*}
Which finishes the proof of  Lemma \ref{lem Kolmogorov}.
\qed\end{proof}

\begin{proof}[Proof of Theorem \ref{local time SHE} \normalfont{\textbf{(ii)}}]
  The proof is a consequence of Lemma \ref{lem Kolmogorov} and \cite[Theorem 3.1]{Berman72}.
\qed\end{proof}
\begin{proof}[Proof of Corollary \ref{cor nonHolder}]
  The proof is a simple application of Theorem \ref{non Holder} and Theorem \ref{local time SHE} \normalfont{\textbf{(ii)}}.
\qed\end{proof}
\begin{rem}
  Let $Y=(Y_t)_{t\in[0,T]}$ be an $\R^d$-valued stochastic process which is $\alpha$-LND with $\alpha\in (0,1)$. By the same calculations as in Section \ref{4}, we can get the following results:
  \begin{enumerate}
    \item Assume that $d<\frac{1}{\alpha}$. Then the process $Y$ has a local time $L(\xi,t)$. Moreover, for every fixed $t$, $L(\bullet,t)\in H^{\beta}(\R^d)$ for $\beta<(\frac{1}{\alpha}-d)/2$, here $H^{\beta}(\R^d)$ is the Sobolev space of index $\beta$.
    \item Assume $d<\frac{1}{\alpha}$. The local time of the process $Y$ has a version, denoted by $L(\xi,t)$, which is s jointly continuous in $(\xi,t)$ almost surely, and which is $\gamma$-H\"{o}lder continuous in $t$, uniformly in $\xi$, for all $\gamma<1-d\alpha$: there exist two random variables $\eta$ and $\delta$ which are almost surely finite and positive such that
        $$\sup_{\xi\in \R^d}|L(\xi,t+h)-L(\xi,t)|\leq \eta|h|^{\gamma},$$
        for all $t,t+h\in [0,T]$ and all $|h|<\delta$.
  \end{enumerate}
\end{rem}

\begin{acknowledgements}
The authors would like to acknowledge the comments, questions, and remarks of the referees. Their help improved the quality and clarity of this paper.
\end{acknowledgements}

\section*{Availability of data and materials}
Data sharing not applicable to this article as no datasets were generated or analysed during the current study.
%
%

\bibliographystyle{spmpsci}      
\bibliography{library}   

\begin{thebibliography}{10}
\providecommand{\url}[1]{{#1}}
\providecommand{\urlprefix}{URL }
\expandafter\ifx\csname urlstyle\endcsname\relax
  \providecommand{\doi}[1]{DOI~\discretionary{}{}{}#1}\else
  \providecommand{\doi}{DOI~\discretionary{}{}{}\begingroup
  \urlstyle{rm}\Url}\fi

\bibitem{Adler}
Adler, R.J.: {The Geometry of Random Fields}.
\newblock Wiley, New York (1981)

\bibitem{BallyMilletSanz}
Bally, V., Millet, A., Sanz-Sole, M.: {Approximation and support theorem in
  Holder norm for parabolic stochastic partial differential equations}.
\newblock Ann. Probab. \textbf{23}(1), 178--222 (1995)

\bibitem{BallyPardoux}
Bally, V., Pardoux, E.: {Malliavin calculus for white noise driven parabolic
  SPDEs}.
\newblock Potential Anal. \textbf{9}(1), 27--64 (1998)

\bibitem{Berman69a}
Berman, S.M.: {Harmonic analysis of local times and sample functions of
  Gaussian processes}.
\newblock Trans. Am. Math. Soc. \textbf{143}, 269 (1969)

\bibitem{Berman69b}
Berman, S.M.: {Local times and sample function properties of stationary
  Gaussian processes}.
\newblock Trans. Am. Math. Soc. \textbf{137}, 277 (1969)

\bibitem{Berman72}
Berman, S.M.: {Gaussian sample functions: Uniform dimension and H{\"{o}}lder
  conditions nowhere}.
\newblock Nagoya Math. J. \textbf{46}, 63--86 (1972)

\bibitem{Bermangeneral1983}
Berman, S.M.: Local nondeterminism and local times of general stochastic
  processes.
\newblock Annales de l'I.H.P. Probabilit\'es et statistiques \textbf{19}(2),
  189--207 (1983)

\bibitem{Berman73}
Berman, S.M., Getoor, R.: Local nondeterminism and local times of gaussian
  processes.
\newblock Indiana University Mathematics Journal \textbf{23}(1), 69--94 (1973)

\bibitem{DalangKhoshnevisanNualartmultiplicative}
Dalang, R.C., Khoshnevisan, D., Nualart, E.: {Hitting probabilities for systems
  of non-linear stochastic heat equations with multiplicative noise}.
\newblock Probab. Theory Relat. Fields \textbf{144}(3-4), 371--427 (2009)

\bibitem{DalangNualart}
Dalang, R.C., Nualart, E.: {Potential theory for hyperbolic SPDEs}.
\newblock Ann. Probab. \textbf{32}(3 A), 2099--2148 (2004)

\bibitem{Ehm}
Ehm, W.: {Sample function properties of multi-parameter stable processes}.
\newblock Zeitschrift f{\"{u}}r Wahrscheinlichkeitstheorie und Verwandte
  Gebiete \textbf{56}(2), 195--228 (1981)

\bibitem{GemanHorowitz}
Geman, D., Horowitz, J.: {Occupation densities}.
\newblock Ann. Probab. \textbf{8}(1), 1--67 (1980)

\bibitem{Nourdin}
Kerchev, G., Nourdin, I., Saksman, E., Viitasaari, L.: {Local times and sample
  path properties of the Rosenblatt process}.
\newblock Stoch. Process. their Appl. \textbf{131}, 498--522 (2021)

\bibitem{Higa}
Kohatsu-Higa, A.: {Lower bounds for densities of uniformly elliptic random
  variables on Wiener space}.
\newblock Probab. Theory Relat. Fields \textbf{126}(3), 421--457 (2003)

\bibitem{LouOuyang}
Lou, S., Ouyang, C.: {Local times of stochastic differential equations driven
  by fractional Brownian motions}.
\newblock Stoch. Process. their Appl. \textbf{127}(11), 3643--3660 (2017)

\bibitem{MoretNualart}
Moret, S., Nualart, D.: {Generalization of It{\^{o}}'s formula for smooth
  nondegenerate martingales}.
\newblock Stoch. Process. their Appl. \textbf{91}(1), 115--149 (2001)

\bibitem{Morien}
Morien, P.L.: {The H{\"{o}}lder and the Besov regularity of the density for the
  solution of a parabolic stochastic partial differential equation}.
\newblock Bernoulli \textbf{5}(2), 275--298 (1999)

\bibitem{nolan1989local}
Nolan, J.P.: Local nondeterminism and local times for stable processes.
\newblock Probability theory and related fields \textbf{82}(3), 387--410 (1989)

\bibitem{Nualart}
Nualart, D.: The Malliavin Calculus and Related Topics.
\newblock Springer (2006)

\bibitem{Sanz-Sole}
Sanz-Sol{\'{e}}, M.: {Malliavin Calculus: With Applications to Stochastic
  Partial Differential Equations}.
\newblock EPFL Press (2005)

\bibitem{Walsh}
Walsh, J.B.: {An introduction to stochastic partial differential equations}.
\newblock In: {\'{E}}cole d'{\'{E}}t{\'{e}} Probab. Saint Flour XIV - 1984, pp.
  265--439 (2006)

\bibitem{Watanabe}
Watanabe, S.: Lectures on stochastic differential equations and Malliavin
  calculus, vol. 164.
\newblock Springer Berlin Heidelberg New York (1984)

\end{thebibliography}

%
%

\end{document}